\documentclass{birkjour}
\usepackage[dvips]{epsfig}
\usepackage{amscd}
\usepackage{amssymb}
\usepackage{amsthm}
\usepackage{amsmath}
\usepackage{mathrsfs}
\usepackage{latexsym}
\usepackage{upref}
\usepackage{hyperref}

\theoremstyle{plain}
\newtheorem{thm}{Theorem}[section]
\theoremstyle{plain}

\newtheorem{lem}[thm]{Lemma}

\newtheorem{cor}[thm]{Corollary}

\theoremstyle{definition}
\newtheorem{defi}{Definition}[section]
\newtheorem{rem}{Remark}

{%
\setcounter{enumi}{0}
\renewcommand{\theenumi}{(\textbf{A}.\arabic{enumi})}
\renewcommand{\labelenumi}{\theenumi}
\begin{enumerate}}
{\end{enumerate} }

{
\setcounter{enumi}{0}
\renewcommand{\theenumi}{(\textbf{D}.\arabic{enumi})}
\renewcommand{\labelenumi}{\theenumi}
\begin{enumerate}}
{\end{enumerate} }

\def\e{{\text{e}}}

\numberwithin{equation}{section} \allowdisplaybreaks

\begin{document}

\title[Internal Stabilization of PIDEs]{Internal Stabilization of a Class of Parabolic Integro-Differential Equations:\\ Application to Viscoelastic Fluids}

\author[Sheetal Dharmatti]{Sheetal Dharmatti}

\address{%
School of Mathematics\\
Indian Institute of Science Education and Research (IISER) Thiruvananthapuram\\
Thiruvananthapuram 695016\\
Kerala, INDIA}

\email{sheetal@iisertvm.ac.in}

\author[Utpal Manna]{Utpal Manna}

\address{%
School of Mathematics\\
Indian Institute of Science Education and Research (IISER) Thiruvananthapuram\\
Thiruvananthapuram 695016\\
Kerala, INDIA}

\email{manna.utpal@iisertvm.ac.in}

\author[Debopriya Mukherjee]{Debopriya Mukherjee}

\address{%
School of Mathematics\\
Indian Institute of Science Education and Research (IISER) Thiruvananthapuram\\
Thiruvananthapuram 695016\\
Kerala, INDIA}

\email{debopriya13@iisertvm.ac.in}








\begin{abstract}
In this paper, we prove the stabilizability of abstract Parabolic Integro-Differential Equations (PIDE) in a Hilbert space with decay rate $e^{-\gamma t} $ for certain $\gamma > 0,$ by means of a finite dimensional controller in the feedback form. We determine a linear feedback law which is obtained by solving an algebraic Riccati equation. To prove the existence of the Riccati operator, we consider a linear quadratic optimal control problem with unbounded observation operator.
 The abstract theory of stabilization developed here is applied to specific problems related to viscoelastic fluids, e.g. Oldroyd B model and Jeffreys model. 
\end{abstract}




\maketitle

\section{Introduction}
Mathematical study of control problems for Parabolic Integro-Differential Equations (PIDE) has gained much attention  in recent times due to its applications in fluid flow problems (e.g. in viscoelastic fluids, polymeric fluids), damped harmonic oscillators, heat-flow problems with memory etc.
The imminent prospect of this paper is to provide a general framework for exponential stabilization of the PIDE in abstract form by virtue of finite dimensional feedback controller. In this paper, we consider the following parabolic integro-differential equation in an abstract form 
\begin{align} \label{PIDE1}
\frac{d\textbf{y}}{dt}+ A \textbf{y}(t)+\int_{0}^{t}\beta(t-s)A\textbf{y}(s)
 \,ds&=0,\quad\mbox{for all}\quad t > 0,\\
 \beta^{'}(t)+ \delta \beta(t)&=0,\quad\mbox{for all}\quad t > 0,\label{PIDE11} 
\end{align}
with initial conditions
\begin{align}\label{PIDE12}
 \textbf{y}(0)=\textbf{y}_{0},\quad \beta(0)=1. \quad (\delta > 0)
\end{align}
where $\textbf{y}_{0} \in H$:   a real Hilbert space with norm  $|.|$ and inner product $(.,.)$ . 
Here $\textbf{y}$ is a state variable. 
 $A:D(A) \subset H \rightarrow H$ is assumed to be closed, densely defined,
linear, self-adjoint, positive definite operator with compact resolvent 
$(\lambda I -A)^{-1}$  for some $\lambda \in \rho (A) $, the resolvent set of $A$. 
We will show the exponential stabilizability of the above system via feedback control. In other words,
we will show that there exists $\omega_{0}>0$ such that the system \eqref{PIDE1} - \eqref{PIDE12} is exponentially stable with 
decay rate $\gamma $ for $0<\gamma<\omega_{0}.$

Using the  standard techniques to stabilize the system \eqref{PIDE1} - \eqref{PIDE12},  we  associate with it a controlled system
\begin{align} \label{PIDE2}
\frac{d\textbf{y}}{dt}+ A \textbf{y}(t)+\int_{0}^{t}\beta(t-s)A\textbf{y}(s)ds&=Bu(t),\quad\mbox{for all}\quad t > 0,\\
 \beta^{'}(t)+ \delta \beta(t)&=0,\quad\mbox{for all}\quad t > 0,\label{PIDE3}
\end{align}
with 
\begin{align}\label{PIDE4}
 \textbf{y}(0)=\textbf{y}_{0},\quad \beta(0)=1. \quad (\delta > 0),
\end{align}
where $u:[0,\infty)\longrightarrow U$ represents a control variable, $U$ is assumed to be a Hilbert space and 
$B:U\rightarrow H $ is a bounded linear operator, i.e.,$B\in\mathcal{L}(U,H).$

The main idea of feedback stabilization is to stabilize stationary but possibly unstable solutions 
 of steady state problem associated with the system. In our case, we are considering zero solution of steady state problem which in general, need not be (asymptotically) stable.  
In particular we would like to show that solution of \eqref{PIDE2} - \eqref{PIDE4} satisfies
$$ \| \textbf{y}(t)\| \leq C \exp^{ - \gamma t } \| \textbf{y}(0) \| \; \mbox{ for} \; t \geq 0. $$

Stabilization results for the non-linear parabolic partial differential equations have been actively studied for the past two decades. Feedback stabilization results for general class of non-linear parabolic problems and in particular Navier-Stokes equation using finite dimensional interior controller have been developed by Barbu \cite{Bar}, Triggiani \cite{Triggi} and references therein. Moreover,  boundary stabilization for fluid flow problems have been extensively studied by Barbu \cite{Bar1}, Triggiani \cite{Triggi}, Badra \cite{Bad}, Raymond \cite{raymond1} \cite{raymond2}, to name a few. The main idea of all these works is to design a controller in the feedback form from the solution of an algebraic Riccati equation, such that  the unstable solution trajectories are exponentially stabilized.

As far as PIDE and Volterra integral equations in Banach spaces are concerned, the existence and uniqueness theory is developed using resolvent operators by Grimmer et al. \cite{Gri}, \cite{Grim11}, \cite{GrimPru1}. The resolvent operator is similar to an evolution operator for non-autonomous differential equations in a Banach space. However the  resolvent operator may not be exponentially bounded and hence will not satisfy semigroup property. For more details one can look into \cite{Gri}, \cite{Grim11},  \cite{GrimPru1}, \cite{Des}, \cite{Des1}.
 Desch and Miller \cite{DeschMiller} have studied
  Volterra integro-differential equations in abstract Banach spaces. By introducing concept of essential growth rate for resolvent operators, stability is obtained.  The location of poles of the operator gives the decay rate for stabilization  \cite{DeschMiller}.

The study of abstract PIDE can be applied to specific class of problems namely, viscoelastic fluid flow. The controllability for linear viscoelastic fluid flow problem has been studied in literature recently. Doubova et al \cite{Dou} have studied approximate controllability exploiting unique continuation property. The approximate controllability for the linearized version of Jeffreys model has been investigated by Chowdhury  et. al. \cite{Chow}.  Authors have come to know of recent work on approximate controllability of PIDE by Pani et. al \cite{Kumar}. Detailed references about existence and uniqueness of Oldroyd model and control problems related to it is discussed in Section 7.

This gives a motivation to  study the stabilization of the corresponding non-linear PIDE around unstable solution trajectories of the stationary problem. 
The controllability problem consists of finding a control which steers the system to a particular state in finite time $T$, whereas the stabilzation deals with finding a control,
such that the solution to the closed-loop system is close to the desired trajectory at all times. 

In our work, we have studied feedback exponential stabilization of abstract PIDE and have shown the application of this result to Oldroyd B model and  Jeffreys model,   linearized around zero steady state solution. 
The main contribution of current article are three important results concerning the stabilizability of PIDE and its applications:
\begin{itemize}
\item Existence of finite dimensional control, which will exponentially stabilize system. $\eqref{PIDE1} - \eqref{PIDE12}$ with the decay rate $\gamma$ such that $ 0 < \gamma < \omega_0$.

\item Existence of feedback control of finite dimension by solving  algebraic Riccati equation.

\item Stabilization of Oldroyd B fluid model and Jeffreys model.
\end{itemize}

The existence of finite dimensional controller which would stabilize the abstract PIDE, is obtained using  spectral analysis of the corresponding operator. For this, we decouple the system in finite and infinite dimensional subspaces of $H$ such that the finite dimensional part of solution is null controllable and infinite dimensional projection is exponentialy stabilizable. 
The important remark is that the finite dimension of the feedback controller is minimal and this choice is done depending upon the maximal multiplicity of unstable eigenvalues of the linear equation.
Further, we prove that finite dimensional feedback controller can be obtained by solving algebraic Riccati equation. Towards this result, we associate a linear quadratic control problem with our system. We have studied the problem in the general case where cost functional depends upon fractional power of operator $A$. This makes the observation operator unbounded in nature. 
Using above results, we prove that the Oldroyd B fluid model and Jeffreys model can be stabilized, when linearized around zero unstable solution of corresponding stationary problems. 
The feedback controller is obtained such that it is localized in an open subset of the given domain. To the best of our knowledge, the results of the paper are the first ones providing feedback control laws stabilizing abstract PIDE.

\qquad


The paper is organized as follows. In Section 2, we give some basic definitions, theorems and Lemmas. In particular we have quoted few results from literature about theory of existence of solution for PIDE. Section 3 is devoted to the representation of solution of PIDE,
using Fourier series expansion. We project the corresponding operator on appropriate finite and infinite dimensional subspaces of $H$. The behaviour of the eigenvalues of the  finite dimensional projection of the operator is explained in this section. 
In Section 4, we prove  exponential stabilizability of the system for decay rate $\gamma; $ $ 0 < \gamma < \omega_0$ using finite dimensional controller. We discuss both the cases;  of  semisimple  and  non-semisimple eigenvalues.
 In Section 5, we show that the finite dimensional controller obtained in the previous section  can be  found in feedback form. This has been done by associating a linear quadratic cost problem and proving the existence of Riccati operator which satisfies the algebraic Riccati equation.
 In Section 6, we deal with  stabilizability of non-homogeneous PIDE. Section 7, is devoted to some  applications of our result to Oldroyd B fluid model and Jeffreys model.
We conclude the paper in Section 8 by giving some further remarks and possible extensions.

\section{Preliminaries}
This section is divided into two parts. In the first part, we report some basic definitions and inequalities, which will be useful in the later sections. In the second part, we discuss briefly about the existence theory of the integro-differential equations of Volterra kind.
\subsection{Some basic definitions and inequalities}
\begin{defi}[Asymptotically stability]
 The equilibrium solution $\textbf{y}_e$ is said to be stable or, more precisely, asymptotically
stable if $$\lim_{t\to\infty}\textbf{y}(t,\textbf{y}_0)=\textbf{y}_e$$
for all $\textbf{y}_0$ in a neighbourhood $\mathscr{V}$ of $\textbf{y}_e.$
\end{defi}
\begin{thm} \label{10.02.16.T1}
Let $A$ be a closed and densely defined operator in Hilbert space $H$ with compact resolvent $(\lambda I -A)^{-1}$
for some $\lambda \in \rho(A)$ (the resolvent set of $A$). Then the spectrum $\sigma(A)$ consists of isolated eigenvalues
$\{\lambda_{j}\}_{j=1}^{\infty}$ each of finite algebraic multiplicity $m_{j}.$
\end{thm} 
This is a particular version of Riesz-Schauder-Fredholm theorem.
For further details, see Yosida \cite{Yosida}, page 283.
\begin{rem}
Using Riesz-Fredholm theory we can conlcude that $A$ has a countable set of real eigenvalues $\lambda_{j}$ each of which
is of finite algebraic multiplicity $m_{j}$ and  corresponding
set of eigenvectors $\phi_{j}$, that is,
\begin{align*}
A\phi_{j}&= \lambda_{j} \phi_{j}\qquad j=1,2,...\\
\mbox{with}\quad \lambda_{j}&\to\infty,\quad j\to \infty.
\end{align*}
For each $\lambda_{j},$ there is a finite number $m_{j}$ of linear independent eigenvectors 
$\{\phi_{j}^{i}\}_{i=1}^{m_{j}} .$ As $A$ is self-adjoint, we note that $\{\phi_{n}\}_{n\in\mathbb{N}}$
forms orthonormal basis of $H$.
\end{rem}
\begin{lem} \label{13.05.15.L1}
 Let $\,\phi \in L^{2}(0,t^*)\,,\, t^{*}>0\,$, then
 \begin{align*}
  \int_{0}^{t^{*}}\int_{0}^{t}\beta(t-s)\phi(s)\,\phi(t) ds\,dt \geqslant 0.
 \end{align*}
\end{lem}
For a proof we refer Sobolevskii (\cite{Sobolev}, page-1601), McLean and Thom\'{e}e \cite{McLean}.\\

\begin{thm} \label{07.03.16.E1}
(Gronwall's lemma). Let $ g,h,y $ be three locally integrable non-negative functions on the time interval $[0,\infty)$ such that for all $t \geq 0,$
\begin{align*}
y(t)+G(t) \leq C+ \int_{0}^{t}h(s)ds+\int_{0}^{t} g(s)y(s)ds,
\end{align*}
where $G(t)$ is a non-negative function on $[0,\infty)$ and $C \geq 0$ is a constant. Then,
\begin{align*}
y(t)+G(t) \leq \left(C + \int_{0}^{t}h(s)ds\right)exp\left(\int_{0}^{t}g(s)ds\right).
\end{align*}
\end{thm}

Now we state the Interpolation inequality for fractional powers.
\begin{lem} \label{15.02.16.L1}
Interpolation inequality: Let $F : D(F) \subseteq X\to X$ be a positive self-adjoint operator
in the Hilbert space $X$, and let $0 \leqslant \theta \leqslant 1$. Then
\begin{align*}
\|F^{\theta}v\|_{X} \leqslant \|Fv\|_{X}^{\theta}\|v\|_{D(F)}^{1-\theta}
\leqslant \theta\|Fv\|_{X}+(1-\theta)\|v\|_{D(F)}\quad \forall  v \in D(F).
\end{align*}
\end{lem}
For proof see Sohr \cite{Sohr} (page 99, Lemma 3.2.2).

\subsection{Existence of solution for integro-differential equations}\label{25.02.16.S1}
To deal with system \eqref{PIDE2}--\eqref{PIDE4}, we need the existence theory for parabolic integro-differential equations. Detailed study of such equations can be found in [\cite{Cordu}, p.235-245],[\cite{GrimPru1},\cite{Des}, \cite{Gri}, \cite{Grim11}]. In this subsection, we quote few results which are relevant to our model. \\
Consider the integro-differential equation of the form:
\begin{align} \label{24.02.16.E1}
x'(t)&=A_1x(t)+\int_{0}^{t}K(t-s)x(s)\,ds+f(t)\\\label{30.03.16.E3}
x(0)&=x_{0} \in D(A_1) \subset X.
\end{align}
where $X$ stands for a Banach space (real or complex).
The following hypothesis attached to the operators $A_1$ and $K(t),t\geq 0$ are :
\begin{itemize}
\item[(i)] $A_1$ is the infinitesimal generator of a semigroup of bounded linear operators acting in $X.$ 
Since $A_1$ is closed, $D(A_1)$ can be organized as a Banach space with the graph norm:
$x \rightarrow |x|+|A_1x|.$ This Banach space will be denoted by $(Y,|.|_{Y}).$
\item[(ii)] $\{K(t);t \geq 0\}$ is a family of bounded linear operators from $Y$ into $X.$
\item[(iii)] For every $x \in Y,$ the function $(Kx)(t)=K(t)x$ is Bochner measurable 
(from $[0,\infty)$ to $X.$)
\item[(iv)]Let
\begin{align*}
|K(\cdot)| \in L^{1}[0,\infty),
\end{align*}
where $|K(t)|$ means the usual norm of the (bounded) linear operator $K(t)$ from $Y$ into $X.$
\end{itemize}
\begin{defi}
 Let $x_0 \in Y.$ A solution of \eqref{24.02.16.E1}-\eqref{30.03.16.E3} is a function that belongs to $C([0,\infty),Y)\cap C^1([0,\infty),X)$
 so that $x(0)=x_0$ and  \eqref{24.02.16.E1}-\eqref{30.03.16.E3} is satisfied for all $t\geqslant 0.$
\end{defi}
For detailed study, see Desch et al. \cite{Des}.\\
Let us now define the resolvent operator corresponding to  \eqref{24.02.16.E1}-\eqref{30.03.16.E3}.
\begin{defi}
A family $\{R(t);t\geq 0\}$ of bounded 
linear operators on $X$ is called a resolvent operator for  \eqref{24.02.16.E1}-\eqref{30.03.16.E3}; if the following conditions are satisfied:
\begin{itemize}
\item[a1] $R(0)=I,$ the identity operator of $X.$
\item[a2] For any $x \in X,$ the map $t \rightarrow R(t)x$ is continuous on $[0,\infty).$
\item[a3] For any $x \in Y,$ the map $t \rightarrow R(t)x$ belongs to $C([0,\infty),Y) \cap C^{1}([0,\infty),X)$ and verifies 
\begin{align*}
R'(t)x=A_1R(t)x+\int_{0}^{t}K(t-s)R(s)x\,ds.
\end{align*}
\item[a4] For any $x \in Y,$ the following equation holds on $[0,\infty):$
\begin{align*}
R'(t)x=R(t)A_1x+\int_{0}^{t}R(t-s)K(s)x\,ds.
\end{align*}
\end{itemize}
\end{defi}
The resolvent operator satisfies a number of properties reminiscent of a semigroup however it does not satisfy an evolution or semigroup property.\\
We see from the above definition that existence of resolvent operator of  \eqref{24.02.16.E1}-\eqref{30.03.16.E3} yields us a representation of
the solution of the equation  \eqref{24.02.16.E1}-\eqref{30.03.16.E3} as
\begin{align}\label{24.02.16.E2}
 x(t)=R(t)x_0+\int_{0}^{t}R(t-s)f(s)\,ds\quad\mbox{for all}\quad t\in[0,\infty).
\end{align}
Let us define, $g:[0,\infty)\times[0,\infty)\to X$ as:
\begin{align}\label{06.03.16.E1}
g(t,s)=\int_{0}^{t}K(s+t-u)x(u)\,du+f(t+s)
\end{align}
Let $D(\mathcal{A})=Y \times H^{1}([0,\infty),X)\subset X \times L^{2}([0,\infty),X).$
Hence,  \eqref{24.02.16.E1}-\eqref{30.03.16.E3} and $(\ref{06.03.16.E1})$ together suggest the pair
$(x,g)$ can be regarded as the solution of the system
\begin{align*}
(x',g')=\mathcal{A}(x,g),
\end{align*}
with the initial condition
\begin{align*}
(x(0),g(0))=(x^{0},f),
\end{align*}
where
\begin{align} \label{25.02.16.E10}
\mathcal{A}(x,g)=\left(A_1x+g(.,0),Kx+\frac{d}{ds}g\right)
\end{align}
\begin{thm}\label{24.02.16.T1}
Let $\mathcal{A}(x,g)$ be given by $(\ref{25.02.16.E10}).$ Assume that the map 
$t \rightarrow K(t)x$ is of bounded variation on $[0,\infty),$ for any $x \in Y.$ 
Then, under the assumptions $(i),(ii),(iii),(iv),$ 
there exists a resolvent operator for the equation  \eqref{24.02.16.E1}-\eqref{30.03.16.E3}.
\end{thm}
For further details, see  Page-242, Theorem 5.3.3 of Cordu \cite{Cordu}.
\begin{thm}\label{06.02.16.T1}
There is at most one resolvent operator for  \eqref{24.02.16.E1}-\eqref{30.03.16.E3}.
\end{thm}
For proof see Theorem 2 of Grimmer and Pr\"{u}ss \cite{GrimPru1}.
\begin{thm}\label{06.02.16.T2}
 Suppose $R(t)$ is resolvent operator for  \eqref{24.02.16.E1}-\eqref{30.03.16.E3}. Let $x_0\in D(A_1)$ and 
 $f\in C([0,\infty),X)\cap L^2([0,\infty),Y)$ or $f\in H^1(0,\infty,X).$ Then, $x(t)$ defined by $(\ref{24.02.16.E2})$ is the 
 solution of  \eqref{24.02.16.E1}-\eqref{30.03.16.E3}.
 
\end{thm}
For further details, see Theorem 3 of Grimmer and Pr\"{u}ss  \cite{GrimPru1}.\\

\textbf{Notation}: Throughout this paper $C$ represents a generic constant.

\section{Spectral analysis and representation of solution for abstract PIDE}
In the first part of this section, we discussed representation of solution and in the next part, we explained the decomposition of the complexified Hilbert space into direct sum of two invariant subspaces by spectral analysis.
\subsection{Existence and regularity of solution:} We primarily focus on the existence, uniqueness and regularity of the solution $\textbf{y}(t),$ after proving existence of the resolvent operator $R(t)$. \\
Let us assume $u\in H^1([0,\infty), U).$ As $B\in\mathcal{L}(U,H),$ we note that $Bu\in H^1([0,\infty), H).$
Taking $A_1=-A,\,\,K(t)=-e^{-\delta t}A,$ and $f(t)=Bu(t),$ we see that all the hypothesis $(i),(ii),(iii),(iv)$ are satisfied. Hence, using Theorem \ref{24.02.16.T1}, Theorem \ref{06.02.16.T1} and Theorem \ref{06.02.16.T2}, we have unique resolvent operator
$R(t):=e^{-t\mathscr{A}(t)}$ and a unique solution 
$\textbf{y}\in C^1([0,\infty),H)$ to the system \eqref{PIDE2}-\eqref{PIDE4} as
\begin{align}\label{06.02.16.E3}
 \textbf{y}(t)=R(t)\textbf{y}_0+\int_{0}^{t}R(t-s)Bu(s)\,ds\quad\mbox{for all}\quad t\in[0,\infty).
\end{align}
Therefore, we have
\begin{align} \label{10.02.16.E4}
\frac{d}{dt}(\textbf{y}(t),\phi)+(\textbf{y},A \phi)+\int_{0}^{t}\beta(t-s)(\textbf{y},A \phi)\,ds=(Bu(t),\phi)
\,\,\mbox{for all}\,\, \phi \in D(A).
\end{align} 
As $\textbf{y}\in L^{2}(0,\infty;H),$ we can write for all $t\geqslant0$
\begin{align} \label{10.02.16.E5}
\textbf{y}(t)=\sum_{n=1}^{\infty}\alpha_{n}(t)\phi_{n}.
\end{align}
where  $\{\phi_{n}\}_{n=1}^{\infty}$ are the eigenfunctions of $A$  and
$\alpha_{n}(t)=(\textbf{y}(t),\phi_{n})$, for all $n\in\mathbb{N}.$
Substituting $\phi=\phi_{n}$ in $(\ref{10.02.16.E4}),$ it yields
\begin{align*}
\alpha_{n}^{'}(t)+\lambda_{n}\alpha_{n}(t)+\lambda_{n}\int_{0}^{t}\beta(t-s)\alpha_{n}(s)\,ds=u_{n}(t) \\
\alpha_{n}(0)=(y_{0},\phi_{n})
\end{align*}
where
\begin{align*}
 u_{n}(t)=(Bu(t),\phi_{n}). 
\end{align*}
Differentiating once more gives,
\begin{align} \label{10.02.16.E7}
\alpha_{n}^{''}(t)+(\lambda_{n}+\delta)\alpha_{n}^{'}(t)+\lambda_{n}(1+\delta)\alpha_{n}(t)=u^{'}_{n}(t)+\delta u_{n}(t).
\end{align}
The corresponding characteristic equation is for equation \eqref{10.02.16.E7} is
\begin{align*}
r^{2}+(\lambda_{n}+\delta)r+\lambda_{n}(1+\delta)=0.
\end{align*}
On solving this, we have two roots $-\mu_{n}^{\pm}$ where,
\begin{align*}
 \mu_n^{\pm}=\frac{1}{2}\left(\lambda_{n}+\delta\pm\sqrt{(\lambda_{n}-\delta)^2-4\lambda_{n}}\right).
\end{align*}
 Hence, solution of $(\ref{10.02.16.E7})$ is given as:
\begin{align}\label{24.02.16.E3}
\alpha_{n}(t)=\frac{1}{(\mu_{n}^+-\mu_n^-)}\left[\left(\mu_n^+\alpha_{n}(0)+\alpha'_{n}(0)\right)e^{-\mu_n^-t}
-\left(\mu_n^-\alpha_{n}(0)+\alpha'_{n}(0)\right)e^{-\mu_{n}^{+}t}\right]\notag\\
+\frac{1}{(\mu_{n}^+-\mu_n^-)}\left[\int_{0}^{t}\left(e^{-\mu_n^-(t-s)}-e^{-\mu_n^+(t-s)}\right)g_{n}(s)ds \right].
\end{align}
where 
\begin{align*}
\alpha_{n}(0)=(\textbf{y}_{0},\phi_{n}),\quad \alpha'_{n}(0)=(Bu(0),\phi_{n})-\lambda_{n}(\textbf{y}_{0},\phi_{n}),\\
\quad\mbox{and}\quad
g_{n}(s)=(Bu'(s),\phi_{n})+ \delta (Bu(s),\phi_{n}) .
\end{align*}
This substituted back in $(\ref{10.02.16.E5}),$ will give the solution $\textbf{y}(t)$ of \eqref{PIDE2}-\eqref{PIDE4}.\\
Now for our convenience, we define the operator 
$$\mathscr{A}:L^2(0,\infty;H)\to L^2(0,\infty;H)$$ by
\begin{align*}
\mathscr{A}(t)\textbf{y}(t)=A \textbf{y}(t)+\int_{0}^{t}\beta(t-s)A\textbf{y}(s)ds\quad\mbox{for all}\quad t> 0.
\end{align*}
\subsection{Spectral analysis in complexified Hilbert space}
Assuming that $\lambda_{n}\to\infty,\,\mu_n^{\pm}$ are all real for large $n.$ 
Taking into account that some of the $\mu_{n}^{\pm}$ might be complex, it is convenient to view $\mathscr{A}$
in the complexified Hilbert space $\tilde{H}=H+iH.$ We denote by $<.,.>$ the scalar product of $\tilde{H}$
and $|.|_{\tilde{H}}$ by it's norm. After simple computation we note that 
\begin{align} \label{10.02.16.E9}
\mu_{n}^{+}\to \infty \quad\mbox{and}\quad \mu_{n}^{-} \to \omega_{0}=\delta+1 \quad\mbox{as}\quad n \to \infty.
\end{align}
\begin{rem}
$\{\mu_{k}^{+}\}_{k=1}^{\infty}$ and $\{\mu_{k}^{-}\}_{k=1}^{\infty}$ are real except possibly for
finitely many complex values. We have the following cases:\\
Case 1. For some specific values of the constant $\delta,$ we can have 
\begin{align*}
\mu_{j}^{+}=\mu_{j}^{-} 
\end{align*}
for some $j.$ This case can occur only for $\delta=\lambda_j\pm2\sqrt{\lambda_j}.$\\
Case 2. It is possible that
\begin{align*}
\mu_{j}^{+}=\mu_{m}^{+}, \quad \mbox{or} \quad \mu_{j}^{-}=\mu_{m}^{-}
\end{align*}
for some $j,m \in \mathbb{N}.$ This can happen if and only if $\lambda_{j}=\lambda_{m},$ 
for some $j,m \in \mathbb{N}.$ In this case, for each fixed $j,$ the multiplicity of $\mu_{j}^{\pm}$ 
is finite and coincides with the multiplicity of $\lambda_{j}.$\\
Case 3. 
It is possible that
\begin{align*}
\mu_{j}^{+}=\mu_{m}^{-}, \quad \mbox{for some}\quad j,m\in\mathbb{N},\,\,j\neq m.
\end{align*}
\end{rem}
For simplicity of exposition, throughout this paper we assume $\delta\neq\lambda_j\pm2\sqrt{\lambda_j}$ and Case 3 does not occur.\\
Let $\gamma$ be such that $0 <\gamma < \omega_{0}.$
Define,
\begin{align} \label{02.03.16.E1}
N_{1}=\sup\{j:\mbox{Re }\mu_{j}^{+} \leq \gamma\},&\quad
N_{2}=\sup\{k:\mbox{Re }\mu_{k}^{-} \leq \gamma\},\notag\\
N&=\max\{N_{1},N_{2}\}
\end{align}
Now using the properties of $A$ and our choice of $\gamma,$ it follows that $N < \infty.$ In the set $\{\lambda_{j}\}_{j=1}^N,$ 
let there be $l$ distinct eigen values with multiplicities $\{m_{k}\}_{k=1}^l$. Also note that $m_1+m_2+\cdots+m_l=N.$
Let
\begin{align} \label{03.03.16.E1}
M=\max\{m_{k}:k=1,...,l\}
\end{align}
We define the linear space $\mathrm{X}_{u}, \mathrm{X}_{s}$ 
\begin{align*}
\mathrm{X}_{u}=\mbox{lin\, span} \{\phi_{j}\}_{j=1}^{N},\quad\mbox{and}\quad
\mathrm{X}_{s}=\mbox{lin\, span} \{\phi_{j}\}_{j=N+1}^{\infty}.
\end{align*}
Let $P_u:\tilde{H} \rightarrow \mathrm{X}_{u}$ be the orthogonal projection
of $\tilde{H}$ onto $\mathrm{X}_{u}.$ Also it is easy to observe that $\mathrm{X}_{u}$ and $\mathrm{X}_{s}$ are $\mathscr{A}(t)$ invariant
subspaces of $H.$

\section{Internal stabilization using finite dimensional control}
In this section we establish the exponential stability of the system \eqref{PIDE1}-\eqref{PIDE12} with a given decay rate
$\gamma$ for $0<\gamma<\omega_0,$  using finite dimensional controller. More precisely,
we decouple the given system into a finite dimensional unstable system  which is null controllable 
and an infinite dimensional $\gamma$-stable part  which is exponentially stable.  
\\Let us denote
$$\mathscr{A}_u=\mathscr{A}|_{L^2(0,\infty;\mathrm{X}_{u})}
\quad\mbox{and}\quad \mathscr{A}_s=\mathscr{A}|_{L^2(0,\infty;\mathrm{X}_{s})}.$$
Let us consider the matrices $$A_{N}=\mbox{diag}\{\lambda_{j}+\delta\}_{j=1}^{N},$$
$$B_{N}=\mbox{diag}\{\lambda_{j}(1+\delta)\}_{j=1}^{N},$$
$$C_{NM}=\{(B\phi_{i}^{*},\phi_{j}^{*})\}_{j=1,i=1}^{N,M},$$
\begin{align*} 
 P_{2N}=
 \begin{pmatrix}
  0_{N}& I_{N} \\
  -B_{N}& -A_{N}
 \end{pmatrix},
 \quad
 Q_{2NM}=
 \begin{pmatrix}
  0_{N}\\
  C_{NM}
 \end{pmatrix}.
\end{align*}
Now by simple computation it is clear that there are $2N$ number of eigenvalues of $P_{2N}$ and they are 
$-\mu_{j}^{+},-\mu_{j}^{-}$ for $j=1,\cdots,N.$ We see from the expression of $\mu_j^\pm$ that for each $\lambda_{j},$ 
there are two eigenvalues $-\mu_j^+$ and $-\mu_j^-$
respectively. Therefore, there are two sets, each of which has $l$ distinct eigenvalues with multiplicities $\{m_{k}\}_{k=1}^l.$
$$\mbox{As}\quad \lambda_{1}=\cdots=\lambda_{m_{1}},\quad\mbox{then}\quad\mu_{1}^+=\cdots=\mu_{m_1}^+,\quad\mbox{and} 
\quad\mu_{1}^-=\cdots=\mu_{m_{1}}^-.$$
Similarly,
$$\mbox{as}\quad \lambda_{m_1+1}=\cdots=\lambda_{m_1+m_{2}},\quad\mbox{then}\quad\mu_{m_1+1}^+=\cdots=\mu_{m_1+m_2}^+,$$
$$\quad\mbox{and} \quad\mu_{m_1+1}^-=\cdots=\mu_{m_1+m_{2}}^-,\cdots,\mbox{as}\quad \lambda_{m_1+\cdots+m_{l-1}+1}=\cdots$$ $$=\lambda_{m_1+\cdots+m_{l}},\quad\mbox{then}
\quad\mu_{m_1+\cdots+m_{l-1}+1}^+=\cdots=\mu_{m_1+\cdots+m_l}^+,$$
$$\quad\mbox{and}\quad\mu_{m_1+\cdots+m_{l-1}+1}^-=\cdots=\mu_{m_1+\cdots+m_{l}}^-.$$ Let us first consider
that the eigenvalues $\{-\mu_{j}^\pm\}_{j=1}^{N}$ are
semisimple. By the theory of Jordan canonical form, there exists an invertible $2N\times2N$ matrix $R_{2N}$
such that 
\begin{align*}
 P_{2N}=R_{2N}^{-1}D_{2N}R_{2N},
\end{align*}
where
\begin{align*}
 D_{2N}=
 \begin{pmatrix}
  M_{N}^+ & 0_{N}\\
  0_{N} & M_{N}^-
 \end{pmatrix},
 \quad
 M_{N}^\pm=\mbox{diag}\{-\mu_{j}^\pm\}_{j=1}^{N}.
\end{align*}
Let $\overline{Q}_{2NM}=R_{2N}Q_{2NM}.$ As $R_{2N}$ is invertible, we note that Rank$(\overline{Q}_{2NM}) = $ Rank$(Q_{2NM}) = M.$
Let $\overline{Q}_{k}$ for $k=1,\cdots,2l$ be the matrices
\begin{align*}
 &\overline{Q}_{1}=\left\{\left(\overline{Q}_{2NM}^*\right)_{ij}\right\}_{i=1,j=1}^{M,m_{1}},\,
 \overline{Q}_{2}=\left\{\left(\overline{Q}_{2NM}^*\right)_{ij}\right\}_{i=1,j=m_{1}+1}^{M,m_1+m_{2}},\\
 &\cdots, 
 \overline{Q}_{l}=\left\{\left(\overline{Q}_{2NM}^*\right)_{ij}\right\}_{i=1,j=m_1+\cdots+m_{l-1}+1}^{M,m_1 +\cdots+m_{l-1}+m_{l}},\\
&\overline{Q}_{l+1}=\left\{\left(\overline{Q}_{2NM}^*\right)_{ij}\right\}_{i=1,j=N+1}^{M,N+m_{1}},\,
 \overline{Q}_{l+2}=\left\{\left(\overline{Q}_{2NM}^*\right)_{ij}\right\}_{i=1,j=N+m_{1}+1}^{M,N+m_1+m_{2}},\\
 &\cdots,\overline{Q}_{2l}=\left\{\left(\overline{Q}_{2NM}^*\right)_{ij}\right\}_{i=1,j
 =N+m_1 +\cdots+m_{l-1}+1}^{M,N+m_1 +\cdots+m_{l-1}+m_{l}}.
\end{align*}
To begin with, we consider the case of semisimple eigenvalues. In later stage we have proved that this assumption is not essential. However this assumption helps us in the step of designing the stabilizing control. Due to this assumption the unstable part of the system reduces to a diagonal finite dimensional differential system. 
We now announce main theorem of this section which deals with existence of finite dimensional controller which stabilizes system\eqref{PIDE2}-\eqref{PIDE4}.
\begin{thm}\label{10.02.16.T2}
Let us assume that 
\begin{align}\label{14.02.16.E11}
\mbox{rank }\overline{Q}_{k}=\mbox{rank }\overline{Q}_{k+l}=m_{k}\quad\mbox{for all}\quad k=1,\cdots,l
\end{align}
and the eigenvalues $\{-\mu_{j}^\pm\}_{j=1}^{N}$ of $P_{2N}$ are semisimple.
Then there is a controller $u$ of the form
\begin{align} \label{PIDE22}
u(t)=\sum_{i=1}^{M}\Phi_{i}v_{i}(t), \quad \forall t\geqslant 0
\end{align}
which stabilizes exponentially system \eqref{PIDE2}-\eqref{PIDE4}.
In other words, the solution $\mathbf{y} \in C^1([0,\infty);\tilde{H})$ of the system \eqref{PIDE2}-\eqref{PIDE4} with control given by $(\ref{PIDE22})$ satisfies,
\begin{align*}
|\mathbf{y}(t)|_{\tilde{H}} \leq C \e^{-\gamma t}|\mathbf{y}_{0}|_{\tilde{H}}, \quad t \geqslant 0.
\end{align*}
Moreover, for any $T>0,$ the controller $v=\{v_{i}\}_{i=1}^{M}$ can be chosen in  $C([0,T),\mathbb{C}^{M})$ such that 
\begin{align*}
\quad v_{i}(t)=-\int_{t}^{T}\beta(t-s)w_{i}(s)\,ds,\,\,
\int_{0}^{T}|w_{i}(t)|^{2}_{M}dt \,&\leq C|\mathbf{y}_{0}|_{\tilde{H}}^{2},\,\,v_{i}(t)=0=w_i(t) 
\end{align*}
$\,\mbox{for}\quad t \geqslant T,$ and $\{\Phi_{i}\}_{i=1}^{M} \,\subset\, D(A)$ is a system of
eigenfunctions. The exact form $\{\Phi_{i}\}_{i=1}^{M} $ is made precise in the proof below. In addition, the controller $v=\{v_i\}_{i=1}^{M}$
 can also be found as a $C^2([0,\infty), \mathbb{C}^M)$ function such that 
 \begin{align*} 
v_{j} &\in C^{2}[0,\infty),\\
|v_{j}(t)|+|v_{j}^{'}(t)|+|v_{j}^{''}(t)| &\leq C e^{-\gamma t}|\mathbf{y}_{0}|_{\tilde{H}}\quad \forall\, t > 0;\, j=1,...,M.
\end{align*}
\end{thm}
\begin{proof}
We represent the solution $\textbf{y}$ to system \eqref{PIDE2}-\eqref{PIDE4} as $\textbf{y}=\textbf{y}_{u}+\textbf{y}_{s},$ where $\textbf{y}_{u} \in \mathrm{X}_{u}$ and $\textbf{y}_{s} \in \mathrm{X}_{s}.$
We can choose a biorthogonal system
$\{\phi_{j}\}_{j=1}^{N}$ and $\{\phi_{j}^{*}\}_{j=1}^{N},$ that is,
\begin{align} \label{12.02.16.E2}
(\phi_{i},\phi_{j}^{*})=\delta_{ij},\, i,j=1,...,N.
\end{align}
For $i=1,...,N, $ we substitute $\Phi_{i}=\phi_{i}^{*}$ in $(\ref{PIDE22}).$
Let us first decouple the system  into a finite dimensional part 
\begin{align} \label{12.02.16.E3}
\frac{d\textbf{y}_{u}}{dt}+ A \textbf{y}_{u}+\int_{0}^{t}\beta(t-s)A\textbf{y}_{u}(s)
 \,ds&=P_u\sum_{i=1}^{M}B\phi^{*}_{i}v_{i}\\\label{28.03.16.E8}
 \textbf{y}_{u}(0)=P_u\textbf{y}_{0}
\end{align}
and an infinite dimensional $\gamma$-stable part,
\begin{align} \label{12.02.16.E4}
\frac{d\textbf{y}_{s}}{dt}+ A \textbf{y}_{s}+\int_{0}^{t}\beta(t-s)A\textbf{y}_{s}(s)
 \,ds&=(I-P_u)\sum_{i=1}^{M}B\phi^{*}_{i}v_{i}\\\label{28.03.16.E9}
 \textbf{y}_{u}(0)=(I-P_u)\textbf{y}_{0}.
\end{align}
In $\mathrm{X}_{u},$ we write $\textbf{y}_{u}=\sum_{n=1}^{N}\alpha_{n}(t)\phi_{n}$ where $\alpha_{n}(t)=(\textbf{y},\phi_{n}).$
Using $(\ref{12.02.16.E2}),$ system \eqref{12.02.16.E3}-\eqref{28.03.16.E8} reduces to:
\begin{align*}
\alpha_{j}^{'}(t)+\lambda_{j}\alpha_{j}(t)+\lambda_{j}\int_{0}^{t}\beta(t-s)\alpha_{j}(s)ds
=(P_u\sum_{i=1}^{M}B\phi^{*}_{i}v_{i}(t),\phi^{*}_{j}),
\end{align*}
i.e.,
\begin{align*}
\alpha_{j}^{'}(t)+\lambda_{j}\alpha_{j}(t)+\lambda_{j}\int_{0}^{t}\beta(t-s)\alpha_{j}(s)\,ds
=\sum_{i=1}^{M}(B\phi^{*}_{i},\phi^{*}_{j})v_{i}(t).                                                      
\end{align*}
with 
\begin{align*} 
\alpha_{j}(0)=(P_uy_{0},\phi^{*}_{j})\,;\quad j=1,\cdots,N.
\end{align*}
Again differentiation with respect to $t,$ gives,  for $j=1,\cdots,N,$
\begin{align*}
\alpha_{j}^{''}(t)+(\lambda_{j}+\delta)\alpha_{j}^{'}(t)+\lambda_{j}(1+\delta)\alpha_{j}(t)=u^{'}_{j}(t)+\delta u_{j}(t)
\end{align*}
where 
\begin{align*}
u_{j}(t)=(Bu(t),\phi_{j}^*)=\sum_{i=1}^{M}(B\phi^{*}_{i},\phi^{*}_{j})v_{i}(t).                                                      
\end{align*}
Let us denote $w_{i}(t)=v_{i}'(t)+\delta v_{i}(t)$ for $i=1,\cdots,M.$
We can write the above system of ordinary differential equations as
\begin{align}\label{13.02.16.E4}
 \begin{pmatrix}
  \alpha_{1}(t)\\
  \cdot\\
  \cdot\\
  \alpha_{N}(t)
 \end{pmatrix}''
 + A_{N}
 \begin{pmatrix}
  \alpha_{1}(t)\\
  \cdot\\
  \cdot\\
  \alpha_{N}(t)
 \end{pmatrix}'
 +B_{N}\begin{pmatrix}
  \alpha_{1}(t)\\
  \cdot\\
  \cdot\\
  \alpha_{N}(t)
 \end{pmatrix}
 = C_{NM} 
 \begin{pmatrix}
  w_{1}(t)\\
  \cdot\\
  \cdot\\
  w_{M}(t)
 \end{pmatrix}.
\end{align} 
Let us denote 
\begin{align*}
 X_N^1(t)= 
 \begin{pmatrix}
  \alpha_{1}(t)\\
  \cdot\\
  \cdot\\
  \alpha_{N}(t)
 \end{pmatrix},
 \quad X_N^2(t)= 
 \begin{pmatrix}
  \alpha_{1}(t)\\
  \cdot\\
  \cdot\\
  \alpha_{N}(t)
 \end{pmatrix}',
 \quad
 W_{M}(t)=
 \begin{pmatrix}
  w_{1}(t)\\
  \cdot\\
  \cdot\\
  w_{M}(t)
 \end{pmatrix}.
\end{align*}
Then, the system $(\ref{13.02.16.E4})$ can be written as
\begin{align}\label{13.02.16.E5}
\begin{pmatrix}
X_N^1(t)\\
X_N^2(t)
\end{pmatrix}'
=P_{2N}
\begin{pmatrix}
X_N^1(t)\\
X_N^2(t)
\end{pmatrix}
+Q_{2NM}W_{M}(t).
\end{align}
Hence in matrix form, the above system can be written as
\begin{align} \label{12.02.16.E8}
X_{2N}'(t)=P_{2N}X_{2N}(t)+Q_{2NM}W_{M}(t)
\end{align}
where 
\begin{align*}
 X_{2N}(t)=
 \begin{pmatrix}
  X_{N}^1(t)\\
  X_{N}^2(t)
 \end{pmatrix}.
\end{align*}
Let us consider the transformation $Z_{2N}(t)=R_{2N}X_{2N}(t).$ Then equation $(\ref{12.02.16.E8})$ can be written as
\begin{align}\label{13.02.16.E8}
 Z_{2N}'(t)=D_{2N}Z_{2N}(t)+\overline{Q}_{2NM}W_{M}(t).
\end{align}
In order to show that system $(\ref{13.02.16.E8})$ is null-controllable, we check the variant of the Kalman controllability criterion. More precisely, we are going to prove that for $t>0,$ the matrix $(\overline{Q}_{2NM})^*e^{tD_{2N}}$ is one-one. Let
\begin{align*}
 \begin{pmatrix}
 x\\
 y
 \end{pmatrix}
 \in\mathbb{R}^{2N}
 \quad\mbox{with}\quad x=\begin{pmatrix}
                          x_{1}\\
                          \cdot\\
                          x_{N}
                         \end{pmatrix}
                         \in\mathbb{R}^N\quad\mbox{and}\quad
                         y=\begin{pmatrix}
                          y_{1}\\
                          \cdot\\
                          y_{N}
                         \end{pmatrix}
                         \in\mathbb{R}^N.
\end{align*}
be such that $$(\overline{Q}_{2NM}^*)e^{tD_{2N}}\begin{pmatrix}
 x\\
 y
 \end{pmatrix}=0. $$
Therefore,
\begin{align*}
 \overline{Q}_{2NM}^*
 \begin{pmatrix}
  e^{-tM_{N}^+}x\\
  e^{-tM_{N}^-}y
 \end{pmatrix}
 =0.
\end{align*}
Hence,
\begin{align*}
 &\sum_{j=1}^{N}\left(\overline{Q}_{2NM}^*\right)_{ij}e^{-t\mu_{j}^{+}}x_{j}
 +\sum_{j=1}^{N}\left(\overline{Q}_{2NM}^*\right)_{ij+N}e^{-t\mu_{j}^{-}}y_{j}=0, 
 \end{align*}
  for $i=1,\cdots,M,\, \forall t \geq 0.$
  This yields,
\begin{align*}
&\sum_{j=1}^{m_{1}}\left(\overline{Q}_{2NM}^*\right)_{ij}e^{-t\mu_{j}^{+}}x_{j}
 +\sum_{j=m_{1}+1}^{m_1+m_{2}}\left(\overline{Q}_{2NM}^*\right)_{ij}e^{-t\mu_{j}^{+}}x_{j}+\cdots+\\
 &\quad\quad+\sum_{j=m_1+\cdots+m_{l-1}+1}^{m_1+\cdots+m_{l}}
 \left(\overline{Q}_{2NM}^*\right)_{ij+N}e^{-t\mu_{j}^{+}}x_{j}+\sum_{j=1}^{m_{1}}
 \left(\overline{Q}_{2NM}^*\right)_{ij+N}e^{-t\mu_{j}^{-}}y_{j}\\
 &\quad\quad\quad\quad\quad\quad
 +\sum_{j=m_{1}+1}^{m_1+m_{2}}\left(\overline{Q}_{2NM}^*\right)_{ij+N}e^{-t\mu_{j}^{-}}y_{j}+\cdots\\ 
 &\quad\quad\quad\quad\quad\quad
 +\sum_{j=m_1+\cdots+m_{l-1}+1}^{m_1+\cdots+m_{l}}
 \left(\overline{Q}_{2NM}^*\right)_{ij+N}e^{-t\mu_{j}^{-}}y_{j}=0.
 \end{align*}
 for $i=1,\cdots,M,\, \forall t \geq 0.$
 Recalling the behaviour of $\mu_{j}^\pm$ as mentioned in Section 3, we have, 
\begin{align*}
  &e^{-t\mu_{m_{1}}^{+}}\sum_{j=1}^{m_{1}}\left(\overline{Q}_{2NM}^*\right)_{ij}x_{j}
 +e^{-t\mu_{m_1+m_{2}}^{+}}\sum_{j=m_{1}+1}^{m_2+m_{2}}\left(\overline{Q}_{2NM}^*\right)_{ij}x_{j}+\cdots+\\
 & +e^{-t\mu_{m_1+\cdots+m_l}^{+}}\sum_{j=m_1+\cdots+m_{l-1}+1}^{m_1+\cdots+m_{l}}
 \left(\overline{Q}_{2NM}^*\right)_{ij+N}x_{j}\\
 &+e^{-t\mu_{m_{1}}^{-}}\sum_{j=1}^{m_{1}}\left(\overline{Q}_{2NM}^*\right)_{ij+N}y_{j}\\
 & +e^{-t\mu_{m_1+m_{2}}^{-}}\sum_{j=m_{1}+1}^{m_1+m_{2}}\left(\overline{Q}_{2NM}^*\right)_{ij+N}y_{j}
 +\cdots\\
&\qquad\qquad 
 +e^{-t\mu_{m_1+\cdots+m_{l}}^{-}}\sum_{j=m_1+\cdots+m_{l-1}+1}^{m_1+\cdots+m_{l}}
 \left(\overline{Q}_{2NM}^*\right)_{ij+N}y_{j}=0.
\end{align*}
Using the fact that $e^{-t\mu_{m_1}^+},e^{-t\mu_{m_1+m_2}^+} ,\cdots, e^{-t\mu_{m_1+\cdots+m_{l}}^+}$ are linearly independent and also
$e^{-t\mu_{m_1}^-},e^{-t\mu_{m_1+m_2}^-} ,\cdots,$ \\
$ e^{-t\mu_{m_1+\cdots+m_{l}}^-}$ are linearly independent
we have for $i=1,\cdots,M;\, k=1,\cdots,l,$
\begin{align*}
 \sum_{j=m_1+\cdots+m_{k-1}+1}^{m_1+\cdots+m_{k}}\left(\overline{Q}_{2NM}^*\right)_{ij}x_{j}=0\\
 \quad\mbox{and}\quad\sum_{j=m_1+\cdots+m_{k-1}+1}^{m_1+\cdots+m_{k}}\left(\overline{Q}_{2NM}^*\right)_{ij+N}y_{j}=0.
\end{align*}
Using $(\ref{14.02.16.E11}),$ we have $x_{j}=0=y_{j}$ for all $j=1,\cdots,N.$ Hence, by Kalman controllability Theorem
there exist control $\{w_{i}\}_{i=1}^M\subset C([0,T];\mathbb{C}^M)$ such that $\textbf{y}_{u}(T)=0.$ Moreover, by the linear 
finite dimensional controllability theory, we know that $\{w_{i}\}_{i=1}^M$ can be chosen in such a way that
 $$\int_{0}^{T}|w_{i}|^2(t)dt\leq C|\textbf{y}_{u}(0)|^2\leq C|\textbf{y}_{0}|_{\tilde{H}}^2\quad\mbox{for all}\quad i=1,\cdots,M.$$
Without loss of generality, we may assume that $w_i(t)=0,\,\,\forall t \geqslant T.$
As $v_i$ satisfies $v_i'(t)+\delta v_i(t)=w_i(t),$ we have 
$$v_i(t)=e^{-\delta t}\left(v_i(0)+\int_{0}^{t}e^{\delta s}w_i(s)ds\right).$$
We assume $v_i(t)=0$ for all $t\geqslant T.$ For  that we need to take $v_i(0)=-\int_0^Te^{\delta s}w_i(s)ds.$
Hence, $$v_i(t)=-\int_t^T e^{-\delta (t-s)}w_i(s)ds=-\int_t^T\beta(t-s)w_i(s)ds.$$
From $(\ref{06.02.16.E3}), (\ref{10.02.16.E5})$ and $(\ref{24.02.16.E3})$ we note that the resolvent 
$R|_{X_s}(t)=e^{-t\mathscr{A}_s(t)}$ satisfies
$\|e^{-t\mathscr{A}_s(t)}\|\leq C e^{-t(\gamma+\epsilon_1)} \leq C e^{-t\gamma}  \quad\mbox{for all}\quad t\geqslant 0,$ for some $\epsilon_1 >0.$
From \eqref{12.02.16.E4}-\eqref{28.03.16.E9} we have 
$$|\textbf{y}(t,\textbf{y}_{0})|\leq C \Big(e^{-\gamma t}|\textbf{y}_{0}|_{\tilde{H}}+\int_{0}^{t}\sum_{i=0}^{M}
|e^{-(t-s)\mathscr{A}_s}B\phi_{i}^{*}v_{i}(s)|ds\Big)\leq C e^{-t \gamma}|\textbf{y}_{0}|_{\tilde{H}}.$$ 
Hence, we have the result.

\end{proof}

\subsection{Case of non-semisimple eigenvalues}
 In the case of non-semisimple eigenvalues, the analysis up to deriving equation \eqref{12.02.16.E8} does not 
 change but the matrix $P_{2N}$ in this case may not be diagonalizable. However, using Jordan theory we can prove that there is a $2N\times 2N$ non singular matrix $\overline{R}_{2N}$ such that $P_{2N}=(\overline{R}_{2N})^{-1}J_{2N}\overline{R}_{2N},$where
 $J_{2N}$ is the Jordan matrix associated with $P_{2N}.$ Thus, the system $(\ref{12.02.16.E8})$  can be written as
 \begin{align*}
 \overline{Z}_{2N}'(t)=J_{2N}\overline{Z}_{2N}(t)+\overline{\overline{Q}}_{2NM}W_{M}(t)
\end{align*}
where $\overline{Z}_{2N}(t)=\overline{R}_{2N}X_{2N}(t)$ and 
$\overline{\overline{Q}}_{2NM}=\overline{R}_{2N}Q_{2NM}.$
Let $\overline{\overline{Q}}_{k}$ for $k=1,\cdots,2l$ be the matrices
\begin{align*}
 &\overline{\overline{Q}}_{1}=\left\{\left(\overline{\overline{Q}}_{2NM}^*\right)_{ij} \right\}_{i=1,j=1}^{M,m_{1}},\,
 \overline{\overline{Q}}_{2}=\left\{\left(\overline{\overline{Q}}_{2NM}^*\right)_{ij}\right\}_{i=1,j=m_{1}+1}^{M,m_1+m_{2}},\cdots,\\ 
 &\overline{\overline{Q}}_{l}=\left\{\left(\overline{\overline{Q}}_{2NM}^*\right)_{ij}\right\}_{i=1,j=m_1+\cdots+m_{l-1}+1}^{M,m_1
 +\cdots+m_{l-1}+m_{l}},\\
 &\overline{\overline{Q}}_{l+1}=\left\{\left(\overline{\overline{Q}}_{2NM}^*\right)_{ij}\right\}_{i=1,j=N+1}^{M,N+m_{1}},\,\\
 &\overline{\overline{Q}}_{l+2}=\left\{\left(\overline{\overline{Q}}_{2NM}^*\right)_{ij}\right\}_{i=1,j=N+m_{1}+1}^{M,N+m_1+m_{2}},
 \cdots,\\
 &\overline{\overline{Q}}_{2l}=\left\{\left(\overline{\overline{Q}}_{2NM}^*\right)_{ij}\right\}_{i=1,j=N+m_1
 +\cdots+m_{l-1}+1}^{M,N+m_1+\cdots+m_{l-1}+m_{l}}.
\end{align*}
If $M=N,$ we will take $\Phi_i=P_u^*\phi_i.$ In this case, det$B\neq 0.$ We can show that the system is null controllable. 
Let $M<N.$ Then, the Jordan matrix $J$ has the following form
$$J^\pm = \{-\mu_{\tilde{m}_1}^\pm E_{\tilde{m}_1}^\pm + H_{\tilde{m}_1},\cdots, 
-\mu_{\tilde{m}_1+\cdots+\tilde{m}_{l}}^\pm E_{\tilde{m}_l}^\pm + H_{\tilde{m}_l}\},$$
where $E_{\tilde{m}_i}$ is a unitary matrix of order $\tilde{m}_i\leqslant m_i$ and $H_{\tilde{m}_i}$ is 
$\tilde{m}_i\times\tilde{m}_i$ matrix of the form
$$H_{\tilde{m}_{i}}=
\begin{pmatrix}
  0&1&0&\cdots&0 \\
  0&0&1&\cdots&0 \\
  0&0&0&\cdots&1 \\
  0&0&0&\cdots&0 
 \end{pmatrix},$$
 Let us set $J_i^\pm = -\mu_{\tilde{m}_1+\cdots+\tilde{m}_{i}}^\pm E_{\tilde{m}_i}^\pm + H_{\tilde{m}_i} ,$ 
 for $i=1,\cdots,l,$ then $J_{2N}$ is the matrix
 $$J_{2N}=
 \begin{pmatrix}
  J^+&0\\
  0&J^-
 \end{pmatrix},\quad
J^\pm=
 \begin{pmatrix}
  J_1^\pm&\quad&\quad&\quad\\
  \quad&J_2^\pm&\quad&0\\
  0&\quad&\cdots&\quad\\
  \quad&\quad&\quad&J_l^\pm
 \end{pmatrix}
$$
Assume for simplicity $\tilde{m}_i=m_i$ (general case follows in a similar way) and note that $m_1+m_2+\cdots+m_l=N.$ \\
We will show that for $t>0,$ the matrix $\overline{\overline{Q}}_{2NM}^*e^{tJ_{2N}^*} $ is one-one.\\
Let,
\begin{align*}
 \begin{pmatrix}
 x\\
 y
 \end{pmatrix}
 \in\mathbb{R}^{2N}
 \quad\mbox{with}\quad x=\begin{pmatrix}
                          x_{1}\\
                          \cdot\\
                          x_{N}
                         \end{pmatrix}
                         \in\mathbb{R}^N\quad\mbox{and}\quad
                         y=\begin{pmatrix}
                          y_{1}\\
                          \cdot\\
                          y_{N}
                         \end{pmatrix}
                         \in\mathbb{R}^N.
\end{align*}
be such that 
\begin{align}\label{22.02.16.E2}
(\overline{\overline{Q}}_{2NM}^*)e^{tJ_{2N}^*}\begin{pmatrix}
 x\\
 y
 \end{pmatrix}=0.
 \end{align}
 Therefore,
\begin{align*}
 \overline{\overline{Q}}_{2NM}^*
 \begin{pmatrix}
  e^{tJ^+}x\\
  e^{tJ^-}y
 \end{pmatrix}
 =0.
\end{align*}
Note that $$e^{t(J^+)^*x}=\begin{pmatrix}
                        e^{t(J_1^+)^*x^1}\\
                        \cdots\\
                        e^{t(J_l^+)^*x^l}
                       \end{pmatrix}
                       \quad\mbox{and}\quad
                       e^{t(J^-)^*x}=\begin{pmatrix}
                        e^{t(J_1^-)^*y^1}\\
                        \cdots\\
                        e^{t(J_l^-)^*y^l}
                       \end{pmatrix}
 $$
 where  $$
 x^k=\begin{pmatrix}
    x_{m_1+\cdots+m_{k-1}+1}\\
    \cdots\\
    x_{m_1+\cdots+m_{k}}
   \end{pmatrix},
   \quad y^k=\begin{pmatrix}
    y_{m_1+\cdots+m_{k-1}+1}\\
    \cdots\\
    y_{m_1+\cdots+m_{k}}
   \end{pmatrix}
   \quad\mbox{for all}\quad k=1,\cdots,l,
$$
with the convention that assuming $m_0=0,$. We note that for $k=1,\cdots,l$
$$e^{tH_{m_k}^*}=I+tH_{m_k}^*+\cdots+
\frac{t^{m_{k}-1}}{({m_{k}-1})!}(H_{m_k}^*)^{m_{k}-1}\\
=\begin{pmatrix}
  1&0&0&0\\
  t&1&0&0\\
  \cdots\\
  \frac{t^{m_{k}-1}}{({m_{k}-1})!}&\frac{t^{m_{k}-2}}{({m_{k}-2})!}&\cdots&1
\end{pmatrix}.
$$
So for $k=1,\cdots,l$
\begin{align*}
e^{tH_{m_k}^*}x^k&=
e^{tH_{m_k}^*}
\begin{pmatrix}
                x_{m_1+\cdots+m_{k-1}+1}\\
                x_{m_1+\cdots+m_{k-1}+2}\\
                \cdot\\
                x_{m_1+\cdots+m_k}
               \end{pmatrix}
\\               
              & =
\begin{pmatrix}
  1&0&0&0\\
  t&1&0&0\\
  \cdots&\cdots&\cdots&\cdots\\
  \frac{t^{m_{k}-1}}{({m_{k}-1})!}&\frac{t^{m_{k}-2}}{({m_{k}-2})!}&\cdots&1
\end{pmatrix}
\begin{pmatrix}
                x_{m_1+\cdots+m_{k-1}+1}\\
                x_{m_1+\cdots+m_{k-1}+2}\\
                \cdots\\
                x_{m_1+\cdots+m_k}
               \end{pmatrix}\\
               &=
               \begin{pmatrix}
                        x_{m_1+\cdots+m_{k-1}+1}\\
                        tx_{m_1+\cdots+m_{k-1}+1}+x_{m_1+\cdots+m_{k-1}+2}\\
                        \cdots\\
                        \frac{t^{m_{k}-1}}{({m_{k}-1})!}x_{m_1+\cdots+m_{k-1}+1}
                        +\frac{t^{m_{k}-2}}{({m_{k}-2})!}x_{m_{m_1+\cdots+k-2}+1}
                        \cdots+x_{m_1+\cdots+m_{k}}
                        \end{pmatrix}.
\end{align*}
Hence, for $k=1,\cdots,l$
\begin{align*}
&e^{t(J_k^+)^*}x^k=e^{-t\mu_{m_1+\cdots+m_{k}}^+(E_{m_k}^+)^*}e^{tH_{m_k}^*}\\
&=e^{-t\mu_{m_1+\cdots+m_{k}}^+(E_{m_k}^+)^*}\\
&\begin{pmatrix}
                        x_{m_1+\cdots+m_{k-1}+1}\\
                        tx_{m_1+\cdots+m_{k-1}+1}+x_{m_1+\cdots+m_{k-1}+2}\\
                        \cdots\\
                        \frac{t^{m_{k}-1}}{({m_{k}-1})!}x_{m_1+\cdots+m_{k-1}+1}
                        +\frac{t^{m_{k}-2}}{({m_{k}-2})!}x_{m_{m_1+\cdots+k-2}+1}
                        \cdots+x_{m_1+\cdots+m_{k}}
                        \end{pmatrix},
\end{align*}
and similarly
\begin{align*}
&e^{t(J_k^-)^*}y^k= e^{-t\mu_{m_1+\cdots+m_{k}}^+(E_{m_k}^-)^*}e^{tH_{m_k}^*}\\\\
&=e^{-t\mu_{m_1+\cdots+m_{k}}^-(E_{m_k}^-)^*}
\\
&\begin{pmatrix}
                        y_{m_1+\cdots+m_{k-1}+1}\\
                        ty_{m_1+\cdots+m_{k-1}+1}+y_{m_1+\cdots+m_{k-1}+2}\\
                        \cdots\\
                        \frac{t^{m_{k}-1}}{({m_{k}-1})!}y_{m_1+\cdots+m_{k-1}+1}
                        +\frac{t^{m_{k}-2}}{({m_{k}-2})!}y_{m_{m_1+\cdots+k-2}+1}
                        \cdots+y_{m_1+\cdots+m_{k}}
                        \end{pmatrix}.
\end{align*}
Hence, using the facts that $e^{-t\mu_{m_1+\cdots+m_{k}}^-(E_{m_k}^-)^*}$'s are linearly independent
and following the same way as in semisimple case, we have from $(\ref{22.02.16.E2}),$ that for $i=1,\cdots, M,$ and
$k=1,\cdots,l,$
\begin{align*}
 \sum_{j=m_1+\cdots+m_{k-1}+1}^{m_1+\cdots+m_{k}}
 \left(\overline{\overline{Q}}_{2NM}^*\right)_{ij}x_{j}=0,\\
 \quad\mbox{and}\quad\sum_{j=m_1+\cdots+m_{k-1}+1}^{m_1+\cdots+m_{k}}\left(\overline{\overline{Q}}_{2NM}^*\right)_{ij+N}y_{j}=0.
\end{align*}
Using the assumption
\begin{align}\label{22.02.16.E3}
\mbox{rank }\overline{\overline{Q}}_{k}=\mbox{rank }\overline{\overline{Q}}_{k+l}=m_{k}
\quad\mbox{for all}\quad k=1,\cdots,l,
\end{align}
we have $x_j=0=y_j$ for all $j=1,\cdots,2N.$ 
Hence, the given system is null controllable. Since $\overline{R}_{2N}$ is non-singular, we can choose $c_{ik}$ in such way that
$\Phi_{i}=\sum_{k=1}^{N}c_{ik}P_u^*\phi_k$ and conditions in $(\ref{22.02.16.E3})$ hold.
Hence we have the general theorem.
\begin{thm}\label{28.02.16.T1}
Under the rank assumptions 
$(\ref{14.02.16.E11})$ for semisimple eigenvalues and $(\ref{22.02.16.E3})$ for non-semisimple eigenvalues, there is a controller $u$ of the form
\begin{align*} 
u(t)=\sum_{i=1}^{M}\Phi_{i}v_{i}(t), \quad \forall t\geqslant 0,
\end{align*}
which stabilizes exponentially system \eqref{PIDE2}-\eqref{PIDE4}.
In other words, the solution $\mathbf{y} \in C^1([0,\infty);\tilde{H})$ of the system \eqref{PIDE2}-\eqref{PIDE4} with control given by $(\ref{PIDE22})$ satisfies,
\begin{align*} 
|\mathbf{y}(t)|_{\tilde{H}} \leq C \e^{-\gamma t}|\mathbf{y}_{0}|_{\tilde{H}}, \quad t \geqslant 0.
\end{align*}
Moreover, for any $T>0,$ the controller $v=\{v_{i}\}_{i=1}^{M}$ can be chosen in $C([0,T),\mathbb{C}^{M})$  such that 
\begin{align*} 
v_{i}(t)=-\int_{t}^{T}\beta(t-s)w_{i}(s)\,ds,\,\,
\int_{0}^{T}|w_{i}(t)|^{2}_{M}dt \,&\leq C|\mathbf{y}_{0}|_{\tilde{H}}^{2},\,\,
v_{i}(t)=0=w_i(t) 
\end{align*}
$\quad \mbox{for}\quad t \geqslant T$ and $\{\Phi_{i}\}_{i=1}^{M} \,\subset\, D(A)$ is a system of
eigenfunctions which is made precise in the proof below. In addition, the controller $v=\{v_i\}_{i=1}^{M}$
 can also be found as a $C^2([0,\infty), \mathbb{C}^M)$ function such that 
 \begin{align*}
v_{j} &\in C^{2}[0,\infty),\\
|v_{j}(t)|+|v_{j}^{'}(t)|+|v_{j}^{''}(t)| &\leq C e^{-\gamma t}|\mathbf{y}_{0}|_{\tilde{H}}\quad \forall\, t > 0;\, j=1,...,M.
\end{align*}
\end{thm}

In the next subsection, from the above construction we will derive  a real valued finite dimensional controller which has a stabilizing effect on the system \eqref{PIDE2}-\eqref{PIDE4}. 
\subsection{Existence of a real valued controller}
As, $A$ is self-adjoint operator, all the eigenvalues and eigenfunctions of $A$ are real.
The above Theorem \ref{10.02.16.T2} can be proven in the real Hilbert space $H$ by taking into account that, 
\begin{align*}
\textbf{y}(t)&=\mbox{Re}\,\textbf{y}(t)+i\mbox{Im}\,\textbf{y}(t),\\
v_{j}&=\mbox{Re}\,v_{j}+i\mbox{Im}\, v_{j}\quad ;j=1,...,M
\end{align*}
Hence, there is a controller 
$u^{*}:[0,\infty) \rightarrow H$ of the form
\begin{align*}
u^{*}(t)=\sum_{i=1}^{M}B\phi_{i}\mbox{Re} v_{i}(t)
\end{align*}
which exponentially stabilizes the real system
\begin{align} \label{15.02.16.E100}
\frac{d\textbf{y}}{dt}+ A \textbf{y}(t)+\int_{0}^{t}\beta(t-s)A\textbf{y}(s)
 \,ds&=u^{*}(t),\,\forall\,t > 0.\\
 \beta^{'}(t)+ \delta \beta(t)&=0.\,\forall\,t > 0.\\\label{28.03.16.E11}
 \textbf{y}(0)=\textbf{y}_{0}, \beta(0)&=1. \quad (\delta > 0)
\end{align}
For $1 \leq M \leq N,$ if the spectrum is semisimple, then
 \begin{align*}
 \Phi_{j}=\phi_{j}^{*},\quad j=1,...,M,
 \end{align*}
 and for general case $\Phi_j\in$ lin span$\{P_u^*\phi_i\}_{i=1}^N$ for $j=1,2,\cdots, M.$\\
 We recall that now onwards everywhere, $|.|$ denotes the norm in H and $(.,.)$ represents it's scalar product. 
Hence we have the following theorem.
\begin{thm} \label{15.02.16.T10}
There is a real controller $u^{*}$ of the form
\begin{align*} 
u^{*}(t)=\sum_{i=1}^{M}B\psi_{i}v_{i}^{*}(t),\quad t>0,
\end{align*}
such that the corresponding solution
\begin{align*} 
\mathbf{y}^{*} \in C^1([0,T];H)\cap L^{2}(0,T;D(A^{\frac{1}{2}}))\cap L^{2}_{Loc}(0,\infty;D(A))\quad \forall\, T>0.
\end{align*}
to, system \eqref{15.02.16.E100}-\eqref{28.03.16.E11}
satisfies the estimate,
\begin{align*} 
|\mathbf{y}^{*}(t)| \leq C e^{-\gamma t}|\mathbf{y}_{0}|\quad \forall\, t \geq 0.
\end{align*}
and
\begin{align*}
v_{j} &\in C^{2}[0,\infty),\\
|v_{j}(t)|+|v_{j}^{'}(t)|+|v_{j}^{''}(t)| &\leq C e^{-\gamma t}|\mathbf{y}_{0}|\quad \forall\, t > 0;\, j=1,...,M.
\end{align*}
If all $\mu_j^\pm$ are semisimple then $\psi_j=\phi_j^*$ and for general case $\psi_j\in  Span \;\{P_u^*\phi_i\}_{i=1}^{M}.$
\end{thm}

\section{Internal stabilization via feedback controller}
This section is devoted to the exponential stabilization of the system \eqref{PIDE2}-\eqref{PIDE4} by means of feedback control derived from associated linear-quadratic optimal control problem. This method is implemented by finding the solution to the algebraic Riccati equation.
It is to be noted that $A^{\epsilon},\, 0< \epsilon < 1,$ is the fractional power of the operator $A$ and $|A^{\epsilon} \textbf{y}|^2=(A^{2\epsilon}\textbf{y},\textbf{y}).$

\begin{thm} \label{26.12.15.T1} Main Theorem:\\
 Let $\alpha \in [\frac{1}{2},\frac{3}{4}].$ Then for each $\gamma,\,0<\gamma<\delta  $, there exists $M$ (as in $(\ref{03.03.16.E1})$), and a linear
self-adjoint operator $R:D(R)\subset H\rightarrow H$ such that for some constants $0<a_{1}<a_{2}<\infty,\,\mbox{and}\, \,C_{1}>0,$ 
we have:
\begin{itemize}
\item[(i)] The following equivalence inequality holds true:
\begin{align} \label{26.12.15.E5}
a_{1}|A^{\alpha-\frac{1}{2}}\mathbf{y}|^{2}\leqslant (R\mathbf{y},\mathbf{y})
\leqslant a_{2}|A^{\alpha-\frac{1}{2}}\mathbf{y}|^{2} 
\quad\mbox{for all}\quad \mathbf{y} \in D(A^{\alpha-\frac{1}{2}}),
\end{align}
\item[(ii)] The operator $R:D(A^{2\alpha-1}) \rightarrow H$ is bounded, i.e.,
\begin{align}\label{14.01.15.E1}
|R\mathbf{y}| \leqslant C_{1}|A^{2\alpha-1}\mathbf{y}| \quad\mbox{for all}\quad \mathbf{y} \in D(A^{2\alpha-1}),
\end{align}
\item[(iii)] $R$ satisfies the following algebraic Riccati equation:
\begin{align}\label{14.01.15.E2}
\left( A \mathbf{y}+\int_{0}^{t}\tilde{\beta}(t-s)A\mathbf{y}(s)
 \,ds-\gamma\mathbf{y},R\mathbf{y}\right)+\frac{1}{2}\sum_{i=1}^{M}\left(B\psi_{i},R\mathbf{y}\right)^{2}
 =\frac{1}{2}|A^{\alpha}\mathbf{y}|^2 
\end{align}
\end{itemize}
 $\quad\mbox{for all}\quad \mathbf{y} \in D(A)$ where $\tilde{\beta}(t)=e^{-(\delta-\gamma)t}.$
Moreover the Feedback controller
\begin{align} \label{07.02.16.E1}
u^{*}(t)=-\sum_{j=1}^{M}(B\psi_{j},R\mathbf{y}(t))\psi_{j}
\end{align}
exponentially stabilizes the linear system
\begin{align} \label{26.12.15.E2}
 \frac{d\mathbf{y}}{dt}+ A \mathbf{y}(t)+\int_{0}^{t}\beta(t-s)A\mathbf{y}(s)
 \,ds&=u^*(t),t\geqslant 0.\\
 \mathbf{y}(0)&=\mathbf{y}_0
\end{align}
that is, the solution $\mathbf{y}^*(t)$ to the corresponding closed loop system satisfies
\begin{align} \label{14.02.16.E113}
|A^{\alpha-\frac{1}{2}}\mathbf{y}^*(t)| &\leqslant C \e^{-\gamma t}|A^{\alpha-\frac{1}{2}}\mathbf{y}_0|
\end{align}
and
\begin{align}\label{14.01.15.E3}
\int_{0}^{\infty}\e^{2\gamma t}|A^{\alpha}\mathbf{y}^*(t)|^2dt&\leqslant C|A^{\alpha-\frac{1}{2}}\mathbf{y}_0|^2.
\end{align}
\end{thm}
\begin{proof}
Let $\alpha \in [\frac{1}{2},\frac{3}{4}].$ Let us consider the optimization problem
\begin{align*}
\varphi(\textbf{y}_{0})=\min\left\{\frac{1}{2}\int_{0}^{\infty}\left( |A^{\alpha}\textbf{y}(t)|^2
+|u(t)|_{M}^{2}\right)dt \right\}
\end{align*}
 subject to $u \in L^{2}(0,\infty;\mathbb{R}^{M})$ and 
\begin{align} \label{07.02.16.E2}
\partial_{t}\textbf{y}+ A \textbf{y}+\int_{0}^{t}\tilde{\beta}(t-s)A\textbf{y}(s)
 \,ds-\gamma\textbf{y}&=\sum_{i=1}^{M}B\psi_{i}u_{i}(t) \\\label{30.03.16.E1}
 \textbf{y}(0)&=\textbf{y}_{0}.
\end{align}
 Let us define $Fu=\sum_{i=1}^{M}B\psi_{i}u_{i}$ for $u=(u_{1},\cdots,u_{M})\in\mathbb{R}^{M}.$ Here $|.|_M$ denotes the Euclidean norm in $\mathbb{R}^{M}$. By previous theorem, there exists an admissible pair $(\textbf{y},u)$ with $\textbf{y} \in  
 L^{2}(0,\infty;H)\cap L^{2}_{Loc}(0,\infty;D(A))$. Indeed, by \eqref{07.02.16.E2}-\eqref{30.03.16.E1} we get the following $a\,priori$ estimates
 \begin{align*}
 \frac{1}{2}\frac{d}{dt}|\textbf{y}(t)|^{2}+|A^{\frac{1}{2}}\textbf{y}(t)|^{2}
 +\int_{0}^{t}\tilde{\beta}(t-s)(A\textbf{y}(s),\textbf{y}(t))ds\, \leqslant \gamma|\textbf{y}(t)|^{2}\\
 +\frac{|Fu(t)|^{2}}{2}+\frac{|\textbf{y}(t)|^{2}}{2}.
  \end{align*}
  Now by the definition of $Fu,$ it is clear that
  \begin{align*} 
  |Fu(t)| \leqslant C |u(t)|_{M}.
  \end{align*}
  Therefore,
 \begin{align*}
 \frac{1}{2}\frac{d}{dt}|\textbf{y}(t)|^{2}+|A^{\frac{1}{2}}\textbf{y}(t)|^{2}
 +&\int_{0}^{t}\tilde{\beta}(t-s)(A^{\frac{1}{2}}\textbf{y}(s),A^{\frac{1}{2}}\textbf{y}(t))ds\notag\\
 &\qquad\qquad\leqslant C \left(|u(t)|_{M}^{2}+|\textbf{y}(t)|^{2}\right).
 \end{align*}
 Integrating on $(0,t),$ we obtain for all $t\geqslant 0,$
 \begin{align*}
 \frac{1}{2}|\textbf{y}(t)|^{2}- \frac{1}{2}|\textbf{y}(0)|^{2}+\int_{0}^{t}|A^{\frac{1}{2}}\textbf{y}(s)|^{2}ds
& +\int_{0}^{t}\int_{0}^{\tau}\tilde{\beta}(\tau-s)(A^{\frac{1}{2}}\textbf{y}(s),A^{\frac{1}{2}}\textbf{y}(\tau))ds\,d\tau\\
& \leqslant C \int_{0}^{t}\left(|u(\tau)|_{M}^{2}+|\textbf{y}(\tau)|^{2}\right)d\tau\\
& \leqslant C \int_{0}^{\infty}\left(|u(\tau)|_{M}^{2}+|\textbf{y}(\tau)|^{2}\right)d\tau\\
 &\leqslant C |\textbf{y}_{0}|^{2}.
 \end{align*}
Hence, using Lemma $\ref{13.05.15.L1},$
\begin{align*}
|\textbf{y}(t)|^{2}+2\int_{0}^{t}|A^{\frac{1}{2}}\textbf{y}(s)|^{2}ds
\leqslant C |\textbf{y}_{0}|^{2}\quad \mbox{for all}\quad t\geqslant 0.
\end{align*}
Therefore, $\textbf{y} \in L^2(0,\infty;D(A^\frac{1}{2})) \cap L^{\infty}(0,\infty;H).$ Let $ 0 \leqslant \theta \leqslant 1.$
Multiplying \eqref{07.02.16.E2}-\eqref{30.03.16.E1} with $A^{\theta} \textbf{y}(t)$ we obtain,
\begin{align} \label{13.02.16.E99}
\frac{1}{2}\frac{d}{dt}(A^{\theta}\textbf{y}(t),\textbf{y}(t))+(A\textbf{y}(t),A^{\theta}\textbf{y}(t))
 +\int_{0}^{t}\tilde{\beta}(t-s)(A\textbf{y}(s),A^{\theta} \textbf{y}(t))ds \notag\\
 -\gamma(A^{\theta}\textbf{y}(t),\textbf{y}(t))
 =(Fu(t),A^{\theta} \textbf{y}(t)).
 \end{align} Thus,
 \begin{align}\label{30.03.16.E2}
 \frac{1}{2}\frac{d}{dt}|A^{\frac{\theta}{2}} \textbf{y}(t)|^{2}+|A^{\frac{1+\theta}{2}} \textbf{y}(t)|^{2}
 +\int_{0}^{t}\tilde{\beta}(t-s)(A^{\frac{1+\theta}{2}}\textbf{y}(s),A^{\frac{1+\theta}{2}} \textbf{y}(t))ds \notag\\
 \leqslant \gamma|A^{\frac{\theta}{2}} \textbf{y}(t)|^{2}+|Fu(t)||A^{\theta} \textbf{y}(t)|.
 \end{align}
 As $A$ is self-adjoint and $\frac{\theta}{2}\leqslant  \theta \leqslant \frac{1+\theta}{2},$
 using the interpolation inequality Lemma \ref{15.02.16.L1}, we get
\begin{align}\label{PIDE1200}
|A^{\theta} \textbf{y}| &\leqslant |A^{\frac{1+\theta}{2}} \textbf{y}|^{\theta}|A^{\frac{\theta}{2}} \textbf{y}|^{1-\theta}
\leqslant \theta |A^{\frac{1+\theta}{2}} \textbf{y}|+(1-\theta)|A^{\frac{\theta}{2}} \textbf{y}|.
\end{align} 
Using $(\ref{PIDE1200})$ we have from $(\ref{30.03.16.E2}),$ 
\begin{align*}
 \frac{1}{2}\frac{d}{dt}|A^{\frac{\theta}{2}} \textbf{y}(t)|^{2}&+|A^{\frac{1+\theta}{2}} \textbf{y}(t)|^{2}
 +\int_{0}^{t}\tilde{\beta}(t-s)(A^{\frac{1+\theta}{2}} \textbf{y}(s),A^{\frac{1+\theta}{2}} \textbf{y}(t))ds\\
 &\leqslant \gamma|A^{\frac{\theta}{2}} \textbf{y}(t)|^{2}
 +\theta|Fu(t)||A^{\frac{1+\theta}{2}} \textbf{y}(t)|+(1-\theta)|Fu(t)||A^{\frac{\theta}{2}} \textbf{y}(t)|
 \end{align*}
Using Young's inequality, we get
\begin{align*}
 \frac{1}{2}\frac{d}{dt}|A^{\frac{\theta}{2}} \textbf{y}(t)|^{2}+|A^{\frac{1+\theta}{2}} \textbf{y}(t)|^{2}
 &+\int_{0}^{t}\tilde{\beta}(t-s)(A^{\frac{1+\theta}{2}} \textbf{y}(s),A^{\frac{1+\theta}{2}} \textbf{y}(t))ds\, \\
 &\leqslant \frac{1}{2}|A^{\frac{1+\theta}{2}} \textbf{y}(t)|^{2}
 +C \left(|u(t)|_{M}^{2}+|A^{\frac{\theta}{2}} \textbf{y}(t)|^{2}\right).
\end{align*}
Now again using the interpolation inequality Lemma $\ref{15.02.16.L1}$ for 
$0 \leqslant \frac{\theta}{2} \leqslant \frac{1}{2},$ we have 
\begin{align} \label{14.02.16.100}
|A^{\frac{\theta}{2}} \textbf{y}| \leqslant \theta |A^{\frac{1}{2}} \textbf{y}|^{2} + (1-\theta)|\textbf{y}|^{2}.
\end{align}
Hence, using $(\ref{14.02.16.100})$ we have,
\begin{align*}
 \frac{1}{2}\frac{d}{dt}|A^{\frac{\theta}{2}} \textbf{y}(t)|^{2}+\frac{1}{2}|A^{\frac{1+\theta}{2}} \textbf{y}(t)|^{2}
 +\int_{0}^{t}\tilde{\beta}(t-s)(A^{\frac{1+\theta}{2}} \textbf{y}(s),A^{\frac{1+\theta}{2}} \textbf{y}(t))ds\, \\
 \leqslant C \left(|u(t)|_{M}^{2}+|A^{\frac{1}{2}} \textbf{y}(t)|^{2}+|\textbf{y}(t)|^{2}\right).
\end{align*}
Integrating on $(0,t)$ and using Lemma $\ref{13.05.15.L1}$ we get, for $0 \leqslant \theta \leqslant 1,$
\begin{align} \label{14.02.16.101}
|A^{\frac{\theta}{2}} \textbf{y}(t)|^{2}+\int_{0}^{t}|A^{\frac{1+\theta}{2}} \textbf{y}(s)|^{2}ds\,
\leqslant C\left(|A^{\frac{\theta}{2}} \textbf{y}_{0}|^{2} + |\textbf{y}_{0}|^{2}\right)
\quad\mbox{for all}\quad t \geqslant 0.
\end{align}
Using Young's inequality, 
 from $(\ref{13.02.16.E99})$ we note that
\begin{align*}
\frac{1}{2}\frac{d}{dt}|A^{\frac{\theta}{2}} \textbf{y}(t)|^{2}&+|A^{\frac{1+\theta}{2}} \textbf{y}(t)|^{2}
 +\int_{0}^{t}\tilde{\beta}(t-s)(A^{\frac{1+\theta}{2}} \textbf{y}(s),A^{\frac{1+\theta}{2}} \textbf{y}(t))ds\,\\
 &\geqslant \gamma|A^{\frac{\theta}{2}} \textbf{y}(t)|^{2}
 -C \left(|Fu(t)||A^{\frac{1+\theta}{2}} \textbf{y}(t)|+|Fu(t)||A^{\frac{\theta}{2}} \textbf{y}(t)|\right)\\
  &\geqslant \gamma|A^{\frac{\theta}{2}} \textbf{y}(t)|^{2}
  -\frac{\gamma}{2}|A^{\frac{\theta}{2}} \textbf{y}(t)|^{2}
  -C\left(|u(t)|_{M}^{2}
 +|A^{\frac{1+\theta}{2}}\textbf{y}(t)|^{2}\right).
\end{align*}
Hence,
\begin{align*}
\frac{1}{2}\frac{d}{dt}|A^{\frac{\theta}{2}} \textbf{y}(t)|^{2} -\frac{\gamma}{2}|A^{\frac{\theta}{2}} \textbf{y}(t)|^{2}
+C|u(t)|_{M}^{2}
+(1+C)|A^{\frac{1+\theta}{2}} \textbf{y}(t))|^{2}\\
+\frac{1}{2}\int_{0}^{t}\tilde{\beta}(t-s)|A^{\frac{1+\theta}{2}} \textbf{y}(t))|^{2}ds
+\frac{1}{2}\int_{0}^{t}\tilde{\beta}(t-s)|A^{\frac{1+\theta}{2}} \textbf{y}(s))|^{2}ds \geqslant 0.
\end{align*}
We note that 
\begin{align}\label{14.02.16.E6}
 \int_{s}^{t}\tilde{\beta}(\tau-s)d\tau =\frac{1}{\delta-\gamma}\left(1-e^{-(\delta-\gamma)(t-s)}\right)\leqslant \frac{1}{\delta-\gamma}
\end{align} and this yields
\begin{align*}
\frac{1}{2}\frac{d}{dt}\{e^{-\gamma t}|A^{\frac{\theta}{2}} \textbf{y}(t)|^{2}\}+C|u(t)|_{M}^{2}
+(1+C+\frac{1}{2( \delta-\gamma)})|A^{\frac{1+\theta}{2}} \textbf{y}(t))|^{2}\\
+\frac{1}{2}\int_{0}^{t}\tilde{\beta}(t-s)||A^{\frac{1+\theta}{2}} \textbf{y}(s))|^{2}ds\geqslant 0.
\end{align*}
Integrating on $(0,t)$ we have,
\begin{align}\label{14.02.16.102}
\frac{1}{2}e^{-\gamma t}|A^{\frac{\theta}{2}} \textbf{y}(t)|^{2}-&\frac{1}{2}|A^{\frac{\theta}{2}} \textbf{y}_{0}|^{2}
+C\int_{0}^{t}|u(\tau)|_{M}^{2}d\tau  \notag\\ 
+&(1+C+\frac{1}{2 \delta})\int_{0}^{t}|A^{\frac{1+\theta}{2}} \textbf{y}(\tau))|^{2} d\tau\notag\\
&+\frac{1}{2}\int_{0}^{t}\int_{0}^{\tau}\tilde{\beta}(t-s)|A^{\frac{1+\theta}{2}} \textbf{y}(s))|^{2}dsd\tau\geqslant 0.
\end{align}
Now using Tonelli's theorem and $(\ref{14.02.16.E6}),$ one can see that
\begin{align*}
\int_{0}^{t}\int_{0}^{\tau}\tilde{\beta}(t-s)|A^{\frac{1+\theta}{2}} \textbf{y}(s)|^{2}ds\,d\tau
&=\int_{0}^{t}\int_{s}^{t}\tilde{\beta}(\tau-s)|A^{\frac{1+\theta}{2}} \textbf{y}(s)|^{2}d\tau\,ds\\
&=\int_{0}^{t}|A^{\frac{1+\theta}{2}} \textbf{y}(s)|^{2}
\left(\int_{s}^{t}\tilde{\beta}(\tau-s)d\tau \right)ds\\
&\leq \frac{1}{\delta-\gamma}\int_{0}^{t}|A^{\frac{1+\theta}{2}} \textbf{y}(s)|^{2}ds.
\end{align*}
From $(\ref{14.02.16.101})$ we see that
\begin{align*}
\limsup_{t \rightarrow \infty}e^{-\gamma t}|A^{\frac{\theta}{2}} \textbf{y}(t)|^{2}
\leq \limsup_{t \rightarrow \infty}e^{-\gamma t}\left(|A^{\frac{\theta}{2}} \textbf{y}_{0}|^{2}+C|\textbf{y}_{0}|^{2}\right)=0.
\end{align*}
Using the above and letting $t \to \infty,$ equation $(\ref{14.02.16.102})$ becomes
\begin{align*}
\frac{1}{2}|A^{\frac{\theta}{2}} \textbf{y}_{0}|^{2}
\leqslant C\int_{0}^{\infty}|u(\tau)|_{M}^{2}d\tau
+(1+C+\frac{1}{\delta-\gamma})\int_{0}^{\infty}|A^{\frac{1+\theta}{2}} \textbf{y}(\tau)|^{2} d\tau.
\end{align*}
This yields that
\begin{align*} 
C_{1}|A^{\frac{\theta}{2}} \textbf{y}_{0}|^{2}
&\leqslant  \int_{0}^{\infty}|A^{\frac{1+\theta}{2}} \textbf{y}(t))|^{2} dt+\int_{0}^{\infty}|u(t)|_{M}^{2}dt\notag\\
&\leqslant C_{2} |A^{\frac{\theta}{2}} \textbf{y}_{0}|^{2}+|\textbf{y}_{0}|^{2}.
\end{align*}
Substituting $\frac{1+\theta}{2}=\alpha,$
we have
\begin{align*}
C_{1}|A^{\alpha-\frac{1}{2}}\textbf{y}_{0}|^2
\leqslant  \varphi(\textbf{y}_{0}) \leqslant C_{2}\left(|A^{\alpha-\frac{1}{2}}\textbf{y}_{0}|^2+|\textbf{y}_{0}|^{2}\right).
\end{align*}
Then, using result of control theory (see Bensoussan et al. \cite{Ben}, page-486, Theorem 3.1), there exists a linear self-adjoint operator 
$R:D(R)\subset H\rightarrow H$ such that $R \in \mathcal{L}(D(A^{\alpha-\frac{1}{2}}),(D(A^{\alpha-\frac{1}{2}}))^{'})$ with
\begin{align} \label{14.02.16.E106}
\varphi(\textbf{y}_{0})=\frac{1}{2}(R\textbf{y}_{0},\textbf{y}_{0}), \quad \mbox{for all}\quad\textbf{y}_{0} 
\in D(A^{\alpha-\frac{1}{2}}).
\end{align}
By the dynamic programming principle (see Barbu \cite{Bar2}, page-190, Theorem 2.1), for each $T >0,$ the solution $(\tilde{\textbf{y}},\tilde{u})$  is the solution to the optimal control problem
\begin{align} \label{14.02.16.E107}
\min\left\lbrace \frac{1}{2}\int_{0}^{T}\left( |A^{\alpha}\textbf{y}(t)|^2
+|\textbf{u}(t)|^{2}\right)dt +\varphi(\textbf{y}(T));\,(\textbf{y},u)
\,\mbox{subject to}\eqref{07.02.16.E2}\mbox{-}\eqref{30.03.16.E1} \right\rbrace.
\end{align}
Hence, by the maximum principle (see Barbu \cite{Bar2}, page 190, Theorem 2.1), we have
\begin{align*}
\tilde{u}(t)=\{(q_{T}(t),B\psi_{i})\}_{i=1}^{M} \quad\mbox{a.e.}\quad t \in (0,T),
\end{align*}
where $q_{T}$ is the solution of dual backward equation
\begin{align*}
q_{T}'(t)-Aq_{T}(t)-\int_{t}^T\tilde{\beta}(s-t)Aq_{T}(s)ds+\gamma q_{T}(t)&
 =A^{2\alpha}\tilde{\textbf{y}}(t),\,\mbox{for all}\,t\in(0,T)\notag\\
 q_{T}(T)&=-R\tilde{\textbf{y}}(T)
\end{align*}
Since $T$ is arbitrary we have 
\begin{align*}
 q_{T}(t)=-R\tilde{\textbf{y}}(t)\quad \mbox{for all}\quad t\geqslant 0
\end{align*}
and therefore
\begin{align*}
\tilde{u}(t)=\{-(R\tilde{\textbf{y}}(t),B\psi_{i})\}_{i=1}^{M} \quad\mbox{a.e.}\quad t \in (0,T).
\end{align*}
Let $\textbf{y}_{0} \in D(A^{2\alpha-1})$ be arbitrary but fixed. First, we will show that $R\textbf{y}_{0}\in H.$
Indeed, using $(\ref{14.02.16.101}), \,(\mbox{for}\,\, \theta\,=\,4\alpha-2),$ we have,
\begin{align}\label{20.02.16.E1}
|A^{2\alpha-1} \tilde{\textbf{y}}(t)|^{2}+\int_{0}^{t}|A^{2\alpha
-\frac{1}{2}} \tilde{\textbf{y}}(s)|^{2}ds\,
\leqslant |A^{2 \alpha-1} \textbf{y}_{0}|^{2} + C_{1}|\textbf{y}_{0}|^{2}
\quad\mbox{for all}\quad t \geqslant 0.
\end{align}
By taking $z(t)=A^{1-2\alpha}q_{T}(t),$ we see that $z(t)$ satisfies 
\begin{align*} 
\frac{dz}{dt}(t)-Az(t)-\int_{t}^{T}\tilde{\beta}(s-t)Az(s)ds+\gamma z(s)=A\textbf{y}^{*}(t),\quad \mbox{in}\quad (0,T]
\end{align*}
and this yields (by multiplying with $A^{4\alpha-2}z$),
\begin{align*}
\frac{1}{2}\frac{d}{dt}|A^{2\alpha-1}z|^{2}-|A^{2\alpha-\frac{1}{2}}z|^{2}
-\int_{t}^{T}\tilde{\beta}(s-t)(Az(s),A^{4\alpha-2}z(t))ds+\gamma|A^{2\alpha-1}z|^{2}\\
=(A\tilde{\textbf{y}},A^{4\alpha-2}z(t))
=(A^{2\alpha-\frac{1}{2}}\tilde{\textbf{y}},A^{2\alpha-\frac{1}{2}}z(t))\\
\geqslant -|A^{2\alpha-\frac{1}{2}}\tilde{\textbf{y}}||A^{2\alpha-\frac{1}{2}}z|.
\end{align*}
Using Young's inequality, we get
\begin{align*}
&\frac{1}{2}\frac{d}{dt}|A^{2\alpha-1}z|^{2}-|A^{2\alpha-\frac{1}{2}}z|^{2}
-\int_{t}^{T}\tilde{\beta}(s-t)(A^{2\alpha-\frac{1}{2}}z(s),A^{2\alpha-\frac{1}{2}}z(t))ds\\
&+\gamma|A^{2\alpha-1}z|^{2}
\geqslant -\frac{1}{2}|A^{2\alpha-\frac{1}{2}}z|^{2}-C|A^{2\alpha-\frac{1}{2}}\tilde{\textbf{y}}|,
\end{align*}
\begin{align*}
\frac{1}{2}\frac{d}{dt}|A^{2\alpha-1}z|^{2}-\frac{1}{2}|A^{2\alpha-\frac{1}{2}}z|^{2}&+\gamma|A^{2\alpha-1}z|^{2}
 +C|A^{2\alpha-\frac{1}{2}}\tilde{\textbf{y}}|\\
&\geqslant \int_{t}^{T}\tilde{\beta}(s-t)(A^{2\alpha-\frac{1}{2}}z(s),A^{2\alpha-\frac{1}{2}}z(t))ds.
\end{align*}
Integrating on $(t,T),$ and using Lemma $\ref{13.05.15.L1}$
\begin{align*}
\frac{1}{2}|A^{2\alpha-1}z(T)|^{2}-&\frac{1}{2}|A^{2\alpha-1}z(t)|^{2}
-\frac{1}{2}\int_{t}^{T}|A^{2\alpha-\frac{1}{2}}z(s)|^{2}ds\\
&+\gamma \int_{t}^{T}|A^{2\alpha-1}z(s)|^{2}ds+C\int_{t}^{T}|A^{2\alpha-\frac{1}{2}}\tilde{\textbf{y}}|^{2}ds\\
&\geqslant \int_{t}^{T} \int_{\tau}^{T}\tilde{\beta}(s-\tau)(A^{2\alpha-\frac{1}{2}}z(s),A^{2\alpha-\frac{1}{2}}z(t))ds\, d\tau
\geqslant 0.
\end{align*}
Hence,
\begin{align*}
|A^{2\alpha-1}z(t)|^{2}&+\int_{t}^{T}|A^{2\alpha-\frac{1}{2}}z(s)|^{2}ds\notag\\
 &\leqslant |A^{2\alpha-1}z(T)|^{2} 
  + C\int_{t}^{T}\left(|A^{2\alpha-\frac{1}{2}}\tilde{\textbf{y}}(s)|^{2}+|A^{2\alpha-1}z(s)|^{2}\right)ds.
\end{align*}
Using Gronwall's inequality (Lemma $\ref{07.03.16.E1}$) and using $(\ref{20.02.16.E1})$ we have,
\begin{align*}
&|A^{2\alpha-1}z(t)|^{2}+\int_{t}^{T}|A^{2\alpha-\frac{1}{2}}z(s)|^{2}ds\\
&\qquad\leqslant \left(|A^{2\alpha-1}z(T)|^{2} 
+ C\int_{t}^{T}|A^{2\alpha-\frac{1}{2}}\tilde{\textbf{y}}(s)|^{2}ds\right)e^{C(T-t)},\\
&\qquad\leqslant C e^{CT}\left(|q_{T}(T)|^{2}+ \int_{0}^{\infty}|A^{2\alpha-\frac{1}{2}}\tilde{\textbf{y}}(s)|^{2}ds \right),\\
&\qquad\leqslant   Ce^{CT} \left(|q_{T}(T)|^{2}+|A^{2\alpha-1}\textbf{y}_{0}|^{2}+|\textbf{y}_{0}|^{2} \right).
\end{align*}
Hence, 
\begin{align*}
z \, \in L^{\infty}(0,T;D(A^{2\alpha-1}))\cap C([0,T];H).
\end{align*}
Thus, $z: [0,T]\rightarrow D(A^{2\alpha-1})$ is weakly continuous and therefore, $q_{T} \in C_{W}([0,T];H),$ space of weakly continuous functions. 
This shows that $q_{T}(0) \in H$ which implies $R\textbf{y}_{0}=-q_{T}(0)\in H.$
From the assumptions on $A$ we have the following inclusion 
\begin{align} \label{20.02.16.E3}
  D(A^{2\alpha-1})\subset D(A^{\alpha-\frac{1}{2}})\subset H \simeq H^{'}\subset(D(A^{\alpha-\frac{1}{2}}))^{'}.
\end{align} 
Let $$G(R)=\{(\textbf{y},R\textbf{y}):\textbf{y} \in  D(A^{2\alpha-1})\}\, \subset D(A^{2\alpha-1})\times H$$ 
be the graph of $R$ in $ D(A^{2\alpha-1})\times H.$
In order to prove $(\ref{14.01.15.E1}),$ it is enough to show that $G(R)$ is closed in $ D(A^{2\alpha-1})\times H.$
Let $\{\textbf{y}_{n}\}_{n \geq 1} \subset  D(A^{2\alpha-1})$ be such that 
\begin{align}\label{20.02.16.E4}
 \textbf{y}_{n} \to\tilde{y}\quad\mbox{in}\quad D(A^{2\alpha-1}),
\quad\mbox{and}\quad R\textbf{y}_{n} \to z\quad\mbox{in} \quad H,\quad n\to\infty. 
\end{align}
We will show that $R\tilde{y}=z.$ Using $(\ref{20.02.16.E3}),$ from $(\ref{20.02.16.E4})$ we have
\begin{align*}
 \textbf{y}_{n} \to\tilde{y}\quad\mbox{in}\quad D(A^{\alpha-\frac{1}{2}}),
\quad\mbox{and}\quad R\textbf{y}_{n} \to z\quad\mbox{in} \quad D(A^{\alpha-\frac{1}{2}})',\quad n\to\infty. 
\end{align*}
Also, we know from the existence of $R$ that
$R \in \mathcal{L}\left(D(A^{\alpha-\frac{1}{2}}),(D(A^{\alpha-\frac{1}{2}}))^{'}\right).$ 
So,
\begin{align*}
 \quad R\textbf{y}_{n} \to R\tilde{y}\quad\mbox{in} \quad D(A^{\alpha-\frac{1}{2}})',\quad n\to\infty. 
\end{align*}
Now by the uniqueness of limit, $R\tilde{y}=z.$ By Closed graph Theorem, $R\in \mathcal{L}(D(A^{\alpha-\frac{1}{2}}),H)$ and hence $(\ref{14.01.15.E1})$ is proved.
Now we will show that $R$ is a solution to Riccati equation $(\ref{14.01.15.E2}).$ 
Again by $(\ref{14.02.16.E107})$, we have,
\begin{align} \label{14.02.16.E109}
\varphi(\tilde{\textbf{y}}(t))=\frac{1}{2}\int_{t}^{\infty}\left( |A^{\alpha}\tilde{\textbf{y}}(s)|^2
+|\tilde{u}(s)|_{M}^{2}\right)ds,\quad \mbox{for all}\quad t \geqslant 0.
\end{align}
From basic calculus, equation $(\ref{14.02.16.E106})$ gives us
\begin{align*}
 \nabla \phi(y)\cdot h=(R(y),h)\quad\mbox{for all}\quad y,h\in D(R).
\end{align*}
Therefore,
\begin{align*}
 \frac{d}{dt}\varphi(\tilde{\textbf{y}}(t))=\left(R\tilde{\textbf{y}}(t),\frac{d\tilde{\textbf{y}}}{dt}(t)\right).
\end{align*}
Using $(\ref{14.02.16.E109}),$ we have that
\begin{align*}
\left(R\tilde{\textbf{y}}(t),\frac{d\tilde{\textbf{y}}}{dt}(t)\right)
=-\frac{1}{2} |A^{\alpha}\tilde{\textbf{y}}(t)|^2
-\frac{1}{2}\sum_{i=1}^{M}(R\tilde{\textbf{y}}(t),B\psi_{i})^{2},\quad \forall \, t\geqslant 0.
\end{align*}
Using \eqref{07.02.16.E2}-\eqref{30.03.16.E1}, we have also for all $t\geqslant 0,$
\begin{align*}
\left(R\tilde{\textbf{y}}(t),\frac{d\tilde{\textbf{y}}}{dt}(t)\right)
=-(R\tilde{\textbf{y}}(t), A \tilde{\textbf{y}}+\int_{0}^{t}\tilde{\beta}(t-s)A\tilde{\textbf{y}}(s)
 \,ds-\gamma\tilde{\textbf{y}}(t)-F\tilde{u}(t)).
\end{align*}
These yield for all $t\geqslant 0$
\begin{align*}
-(R\tilde{\textbf{y}}(t), A \tilde{\textbf{y}}+&\int_{0}^{t}\tilde{\beta}(t-s)A\tilde{\textbf{y}}(s)
 \,ds-\gamma\tilde{\textbf{y}}(t))+(F(\tilde{u}(t)),R\tilde{\textbf{y}}(t))\notag\\
 &+\frac{1}{2}\sum_{i=1}^{M}(R\tilde{\textbf{y}}(t),B\psi_{i})^{2}
 +\frac{1}{2} |A^{\alpha}\tilde{\textbf{y}}(t)|^2=0,
\end{align*}
\begin{align*}
 -(R\tilde{\textbf{y}}(t), A \tilde{\textbf{y}}+&\int_{0}^{t}\tilde{\beta}(t-s)A\tilde{\textbf{y}}(s)
 \,ds-\gamma\tilde{\textbf{y}}(t))\notag\\
 &-\frac{1}{2}\sum_{i=1}^{M}(R\tilde{\textbf{y}}(t),B\psi_{i})^{2}
 +\frac{1}{2} |A^{\alpha}\tilde{\textbf{y}}(t)|^2=0.
\end{align*}
Hence, for all $t\geqslant 0,$
\begin{align} \label{14.02.16.E110}
(R\tilde{\textbf{y}}(t), A \tilde{\textbf{y}}+&\int_{0}^{t}\tilde{\beta}(t-s)A\tilde{\textbf{y}}(s)
 \,ds-\gamma\tilde{\textbf{y}}(t))\notag\\
 &+\frac{1}{2}\sum_{i=1}^{M}(R\tilde{\textbf{y}}(t),B\psi_{i})^{2}
 =\frac{1}{2} |A^{\alpha}\tilde{\textbf{y}}(t)|^2,
\end{align}
which  directly implies $(\ref{14.01.15.E2}).$ Let $\tilde{\textbf{y}}^{*}=e^{-\gamma t}\tilde{\textbf{y}}$ 
and $\tilde{u}^{*}=e^{-\gamma t}\tilde{u}.$ 
As $(\tilde{\textbf{y}},\tilde{u})$ satisfy \eqref{07.02.16.E2}-\eqref{30.03.16.E1}, then $(\tilde{\textbf{y}}^{*},\tilde{u}^{*})$ satisfy 
\begin{align*}
 \frac{d\tilde{\textbf{y}}^*}{dt}+ A \tilde{\textbf{y}}^*(t)+\int_{0}^{t}\beta(t-s)A\tilde{\textbf{y}}^*(s)
 \,ds+\sum_{i=1}^{M}(R\tilde{\textbf{y}}^*(t),B\psi_{i})B\psi_{i}=0,t\geqslant 0.
\end{align*}
Multiplying the closed loop system \eqref{07.02.16.E2}-\eqref{30.03.16.E1}
\begin{align} \label{14.02.16.E111}
 \frac{d\tilde{\textbf{y}}}{dt}+ A \tilde{\textbf{y}}(t)+\int_{0}^{t}\tilde{\beta}(t-s)A\tilde{\textbf{y}}(s)
 \,ds-\gamma\tilde{\textbf{y}}(t)+\sum_{i=1}^{M}(R\tilde{\textbf{y}}(t),B\psi_{i})B\psi_{i}=0,t\geqslant 0.
\end{align}
by $R\tilde{\textbf{y}}(t)$ and using $(\ref{14.02.16.E110})$ we get,
\begin{align} \label{14.02.16.E112}
\frac{1}{2}\frac{d}{dt}(R\tilde{\textbf{y}}(t),\tilde{\textbf{y}}(t))
+\frac{1}{2} |A^{\alpha}\tilde{\textbf{y}}(t)|^2
+\frac{1}{2}\sum_{i=1}^{M}(R\tilde{\textbf{y}}(t),B\psi_{i})^{2}=0,\, \forall\, t \geqslant 0.
\end{align}
Since the second and third terms in the above equation are positive, we have
\begin{align*}
\frac{d}{dt}\left((R\tilde{\textbf{y}}(t),\tilde{\textbf{y}}(t))\right)
\leq 0.
\end{align*}
On integration we have,
\begin{align*}
(R\tilde{\textbf{y}}(t),\tilde{\textbf{y}}(t))\leq (R\textbf{y}_{0},\textbf{y}_{0}).
\end{align*}
Hence, using $(\ref{26.12.15.E5}),$ we have
\begin{align*}
|A^{\alpha-\frac{1}{2}}\tilde{\textbf{y}}^*|^2\leq  Ce^{-2\gamma t}|A^{\alpha-\frac{1}{2}}\textbf{y}_0|^2.
\end{align*}
Now, using $(\ref{26.12.15.E5}),$ we have $(\ref{14.02.16.E113}).$
For the other inequality, we note that
\begin{align*}
\frac{d}{dt}\left(R\tilde{\textbf{y}}(t),\tilde{\textbf{y}}(t)\right)
+|A^{\alpha}\tilde{\textbf{y}}(t)|^2 \leq 0,
\end{align*}
which again on integration yields,
\begin{align*}
(R\tilde{\textbf{y}},\tilde{\textbf{y}})-(R\textbf{y}_{0},\textbf{y}_{0})
+\int_{0}^{t}|A^{\alpha}\tilde{\textbf{y}}(s)|^2ds \leq 0,
\end{align*}
using the positivity property of the first term and $(\ref{26.12.15.E5})$ we have
\begin{align*}
\int_{0}^{t}e^{2\gamma s}|A^{\alpha}\tilde{\textbf{y}}^*(s)|^2ds \leq (R\textbf{y}_{0},\textbf{y}_{0})
\leq a_{2} |A^{\alpha-\frac{1}{2}}\textbf{y}_0|^2 .
\end{align*}
Now letting $t\rightarrow \infty,$ by monotone convergence theorem, we have $(\ref{14.01.15.E3}).$

\end{proof}


 \begin{thm} \label{03.04.16.T1}
 Let $\alpha \in [0,\frac{1}{2}).$ Then for each $\gamma,\,0<\gamma<\delta $ and $M$ (as in $(\ref{03.03.16.E1})$),
 there is a bounded linear self-adjoint positive semidefinite operator $\tilde{R}:H\rightarrow H$ such that $\tilde{R}\in\mathcal{L}(H,D(A))$  
and $\tilde{R}$ satisfies the following algebraic Riccati equation:
\begin{align*}
\left( A \textbf{y}+\int_{0}^{t}\tilde{\beta}(t-s)A\textbf{y}(s)
 \,ds-\gamma\textbf{y},\tilde{R}\textbf{y}\right)+\frac{1}{2}\sum_{i=1}^{M}\left(B\psi_{i},\tilde{R}\textbf{y}\right)^{2}
 =\frac{1}{2}|A^{\alpha}\textbf{y}|^2 \\
 \forall \, \textbf{y} \in D(A), 
\end{align*}
where $\tilde{\beta}(t)=e^{-(\delta-\gamma)t}.$
Moreover the feedback controller
\begin{align*} 
u^{*}(t)=-\sum_{j=1}^{M}(\tilde{R}\textbf{y}(t),B\psi_{j})\psi_{j}
\end{align*}
exponentially stabilizes the linear system
\begin{align*} 
 \frac{d\textbf{y}}{dt}+ A \textbf{y}(t)+\int_{0}^{t}\beta(t-s)A\textbf{y}(s)
 \,ds&=u^*(t),t\geqslant 0.\\
 \textbf{y}(0)&=\textbf{y}_0
\end{align*}
that is, the solution $\textbf{y}^*$ to the corresponding closed loop system satisfies
\begin{align*} 
|\textbf{y}^*(t)| &\leqslant C\e^{-\gamma t}|\textbf{y}_0|.
\end{align*}
\end{thm}

 \begin{proof}
Let $\alpha \in [0,\frac{1}{2}).$ Consider the optimization problem
\begin{align*}
\varphi(\textbf{y}_{0})=\min\left\{\frac{1}{2}\int_{0}^{\infty}\left( |A^{\alpha}\textbf{y}(t)|^2
+|u(t)|_{M}^{2}\right)dt \right\}
\end{align*}
 subject to $u \in L^{2}(0,\infty;\mathbb{R}^{M})$ and 
\begin{align*} 
\frac{d\textbf{y}}{dt}+ A \textbf{y}+\int_{0}^{t}\tilde{\beta}(t-s)A\textbf{y}(s)
 \,ds-\gamma\textbf{y}&=\sum_{i=1}^{M}B\psi_{i}u_{i}(t) \\
 \textbf{y}(0)&=\textbf{y}_{0}.\notag
\end{align*}
Proceeding in the similar manner as in the proof of previous theorem, we obtain
\begin{align}\label{03.04.16.E8}
|\textbf{y}(t)|^{2}+2\int_{0}^{t}|A^{\frac{1}{2}}\textbf{y}(s)|^{2}ds
\leqslant C_{2}|\textbf{y}_{0}|^{2}\quad \mbox{for all}\quad t\geqslant 0.
\end{align}
Now due to the interpolation inequality (Lemma \ref{15.02.16.L1}) for 
$0 \leqslant \alpha \leqslant \frac{1}{2},$ we have 
\begin{align} \label{03.04.16.E9}
|A^{\alpha} \textbf{y}|^2 \leqslant 2\alpha |A^{\frac{1}{2}} \textbf{y}|^{2} + (1-2\alpha)|\textbf{y}|^{2}.
\end{align}
Using $(\ref{03.04.16.E8})$ and $(\ref{03.04.16.E9})$ we obtain,
\begin{align*}
 \int_{0}^{t}|A^{\alpha} \textbf{y}|^2 \leqslant \int_{0}^{t} 2\alpha |A^{\frac{1}{2}} \textbf{y}|^{2} 
 +\int_{0}^{t} (1-2\alpha)|\textbf{y}|^{2}\leqslant C |\textbf{y}_0|^2,
\end{align*}
which finally yields
\begin{align*}
 \varphi(\textbf{y}_0)\leqslant C |\textbf{y}_0|^2.
\end{align*}
Rest of the proof will follow as in Theorem \ref{26.12.15.T1}, with minor modifications from place to place.
 \end{proof}

\section{Stabilization under nonzero forcing field} 
We report in this section the stabilization of the distributed control problem
 with the non-zero forcing term $\textbf{f} \in H^{1}((0,T);H),\quad \forall\,\, T>0.$
\\Let us consider the system
\begin{align} \label{PIDE100}
\frac{d\textbf{y}}{dt}+ A \textbf{y}(t)+\int_{0}^{t}\beta(t-s)A\textbf{y}(s)
 \,ds&=\textbf{f},\quad\mbox{for all}\quad t > 0,\\
 \beta^{'}(t)+ \delta \beta(t)&=0,\quad\mbox{for all}\quad t > 0,\label{PIDE111}
\end{align}
with initial condition
\begin{align}\label{PIDE12000}
 \textbf{y}(0)=\textbf{y}_{0},\quad \beta(0)=1. \quad (\delta > 0)
\end{align}
with $\textbf{y}_{0} \in H$
where $\,\textbf{f}\,$ is a given forcing field.
Assume $\textbf{f}\in H^{1}((0,T);H),\,\forall\,T>0.$\\
Using the well-posedness theory mentioned in Section 2, we conclude the existence of a unique solution 
$\textbf{y}\in C^1([0,\infty), H)$
of the above system \eqref{PIDE100} -\eqref{PIDE12000}. Exploiting definition of equilibrium solution, it yields,
 $\textbf{y}_{e} \in D(A)$ is a solution to the steady state equation
\begin{align*} 
(1+\frac{1}{\delta})A\textbf{y}_{e}=\textbf{f}_{e}
\end{align*}  
where $\textbf{f}_{e}=\lim_{t \rightarrow \infty} \textbf{f}(t).$
It is convenient to reduce the stabilization problem around $\textbf{y}_{e}$ to that of 
zero solution by setting $\textbf{y}-\textbf{y}_{e}\Longrightarrow \textbf{y}$ and so, 
to transform \eqref{PIDE100}-\eqref{PIDE12000} into
\begin{align} \label{27.02.16.E2}
\frac{d\textbf{y}}{dt}+ A \textbf{y}(t)+\int_{0}^{t}\beta(t-s)A\textbf{y}(s)
 \,ds&=g(t)+\textbf{f}(t)-\textbf{f}_{e} \quad\mbox{for all}\quad t > 0,\\
 \beta^{'}(t)+ \delta \beta(t)&=0,\quad\mbox{for all}\quad t > 0\\\label{28.03.16.E1}
 \textbf{y}(0)=\textbf{y}_{0}-\textbf{y}_{e}&=\tilde{\textbf{y}}_{0},\\\label{28.03.16.E2}
 \beta(0)&=1, \quad (\delta > 0)
\end{align}
with $g(t)=\frac{1}{\delta}A\textbf{y}_{e}e^{-\delta t}$ for all $t>0.$\\
Consider the following translated control system associated to \eqref{27.02.16.E2}-\eqref{28.03.16.E2} :
\begin{align} \label{27.02.16.E3}
\frac{d\textbf{y}}{dt}+ A \textbf{y}(t)+\int_{0}^{t}\beta(t-s)A\textbf{y}(s)
 \,ds&=g(t)+\textbf{f}(t)-\textbf{f}_{e}+B\tilde{u}(t) \quad\mbox{for all}\quad t > 0,\\\label{28.03.16.E4}
 \beta^{'}(t)+ \delta \beta(t)&=0,\quad\mbox{for all}\quad t > 0\\
 \textbf{y}(0)=\textbf{y}_{0}-\textbf{y}_{e}&=\tilde{\textbf{y}}_{0},\\\label{28.03.16.E5}
 \beta(0)&=1. \quad (\delta > 0)
\end{align}
where $\tilde{u}:[0,\infty)\longrightarrow U$ represents a control and 
$B$ is a bounded linear operator from the control space $U$ to $H$.
\begin{cor}\label{28.02.16.C1}
 Assume that $\mathbf{f}:[0,\infty)\to R(B)\subset H.$ Let the rank assumptions $\eqref{14.02.16.E11}$ for semisimple eigenvalues and $(\ref{22.02.16.E3})$ for non-semisimple eigenvalues hold. Then there exists a controller 
which exponentially stabilizes  system \eqref{27.02.16.E3}-\eqref{28.03.16.E5}. Consequently all the conclusions of Theorem \ref{28.02.16.T1} hold true.

\end{cor}
\begin{proof}
Let $\mathbf{\tilde{f}}:[0,\infty)\to U$ be such that 
 $B\mathbf{\tilde{f}}(t)=\mathbf{f}(t)$ for all $t\geqslant 0$ and 
 there exists $\mathbf{\tilde{\tilde{f}}}\in U$ such that $B\mathbf{\tilde{\tilde{f}}}=\mathbf{f}_e.$
Note that $g \in$  Range $B.$ For $g_{1}=\frac{e^{- \delta t}}{1+\delta} \tilde{\tilde{\textbf{f}}} \in U $ we have $B(g_{1})=g.$ 
Let us define $$u_1(t)=\tilde{u}(t)+\tilde{\textbf{f}}(t)-\left(1-\frac{e^{-\delta t}}{1+\delta}\right)\tilde{\tilde{\textbf{f}}}
\quad\mbox{for all}\quad t>0.$$
Then, system \eqref{27.02.16.E3}-\eqref{28.03.16.E5} becomes  
\begin{align*}
\frac{d\textbf{y}}{dt}+ A \textbf{y}(t)+\int_{0}^{t}\beta(t-s)A\textbf{y}(s)
 \,ds&=Bu_1(t),\quad\mbox{for all}\quad t > 0,\\
 \beta^{'}(t)+ \delta \beta(t)&=0,\quad\mbox{for all}\quad t > 0\\
 \textbf{y}(0)=\textbf{y}_{0}-\textbf{y}_{e}&=\tilde{\textbf{y}_{0}},\\
 \beta(0)&=1. \quad (\delta > 0)
\end{align*}
Now we can apply the Theorem \ref{28.02.16.T1} to conclude the result.
\end{proof}
\begin{rem}
 We note that similar results as in Theorem \ref{26.12.15.T1} and Theorem \ref{03.04.16.T1} hold for the system \eqref{PIDE100}-\eqref{PIDE12000}.
\end{rem}

In the following Section, we will provide applications to viscoelastic fluids of the abstract theory developed in this paper.
\section{Applications to viscoelastic fluids}
\subsection{ Application to  Oldroyd fluid}
Viscoelastic fluids are the kind of fluids which exhibit both viscous and elastic characteristics while undergoing strain. It is known for a long time that such fluids are non-Newtonian in nature and have memory property. One of the most well-known linear viscoelastic fluid model was proposed by J. G. Oldroyd \cite{Old} and is known as Oldroyd-B fluid. This model encompasses majority of viscous, incompressible, non-Newtonian fluids encountered in practice with flows of moderate velocities. For further details of the physical background and mathematical modelling, we refer to Pani et al. \cite{Pani}\cite{Gos}, Joseph \cite{Jos}, Oldroyd \cite{Old} and  references therein.\\
The focus of this Section is to concentrate on the two dimensional Oldroyd model with zero forcing term in a bounded domain $\,O\,$ in $\,\mathbb{R}^{2}\,$ with smooth boundary $\partial O.$ 
We denote the velocity field by 
$\,\textbf{y}\,$ and the pressure field by $\,p.\,$ The system of equations of motion arising in the Oldroyd fluids of order one is:
\begin{align} \label{25.04.15.E1}
 \partial_{t}\textbf{y}+(\textbf{y}\cdot\nabla)\textbf{y}-\mu\Delta\textbf{y}-\int_{0}^{t}\beta(t-\tau)\Delta\textbf{y}(x,\tau)
 \,d\tau+\nabla p = 0 \qquad \mbox{in}\quad O\times(0,T).
\end{align}
\begin{align} \label{25.04.15.E2}
\nabla\cdot\textbf{y}=0 \quad \mbox{in}\quad O\times(0,T)
\end{align}
with initial and boundary conditions,
\begin{align} \label{25.04.15.E3}
 \textbf{y}&=0 \qquad \mbox{in}\quad\partial O\times(0,T)\\\label{28.03.15.E9}
 \textbf{y}&=\textbf{y}_0 \qquad \mbox{in}\quad O\times\{0\}.
\end{align}
Here, $\mu= \frac{2\kappa}{\lambda} >0 $ and the kernel $\beta(t)=\gamma e^{-\delta t}$ with $\gamma=\frac{2}{\lambda}(\nu- \frac{\kappa}{\lambda})>0$ and $\delta=\frac{1}{\lambda}>0.$\\ For further details we refer Goswami and Pani \cite{Gos} and the references therein.

There is considerable amount of work available in the literature regarding the Oldroyd model. Oskolkov \cite{Os} established the global well-posedness of the classical solution in two-dimensions following the celebrated work of Ladyzhenskaya \cite{La} on the solvability of Navier-Stokes equations. Well-posedness theory was further investigated by many other mathematicians (see \cite{AS}, \cite{KO}, to name a few) under different regularity of initial conditions. In three-dimensions, one can at-most expect local-in-time solvability result with arbitrary initial data and global-in-time result for sufficiently small initial data, much like the Navier-Stokes equations. It is worth to note the work of Lions and Masmoudi \cite{LM}, where the authors considered a more general Oldryod model (with much stronger non-linear coupling) and proved the existence of global weak solutions for general initial conditions. 
 
In \cite{Sobolev}, Sobolevskii explained the behaviour of the solution as $t \rightarrow \infty $ under some stabilization conditions like positivity of the first eigenvalue of a self-adjoint spectral problem
introduced therein and H\''{o}lder continuity of the function $\Phi= e^{\delta_{0}t}(f (x, t) - f_{\infty}(x))$, where
$f_{\infty} = \limsup_{t \rightarrow \infty} f$ and $\delta_{0}> 0,$ using energy arguments and positivity of the integral operator.  In \cite{MML}, Marinho et al. established exact controllability for the Oldroyd model in finite-dimensional system using the Hilbert Uniqueness Method in combination with the Schauder's fixed point.

Our aim, in this work, is to design a feedback controller with 
support in an arbitrary open subset $O_0\subset O$ such that the solution $\textbf{y}$ around equilibrium solution $\textbf{y}_{e}=0$ is exponentially stabilized with decay rate $\gamma$ for $0 <\gamma <\delta$. The linearized (around  $\textbf{y}_{e}=0$) control system corresponding to \eqref{25.04.15.E1}-\eqref{28.03.15.E9} is given by
\begin{align} \label{21.02.16.E1}
 \partial_{t}\textbf{y}-\mu\Delta\textbf{y}-\int_{0}^{t}\beta(t-\tau)\Delta\textbf{y}(x,\tau) 
 \,d\tau+\nabla p &= u \,\chi_{O_0}  \qquad \mbox{in}\quad O\times(0,T) \\
\nabla\cdot\textbf{y} &= 0 \quad \mbox{in}\quad O\times(0,T)
\end{align}
with initial and boundary condition,
\begin{align} \label{21.02.16.E3}
 \textbf{y}&=0 \qquad \mbox{in}\quad\partial O\times(0,T)\\\label{28.03.16.E12}
 \textbf{y}&=\textbf{y}_0 \qquad \mbox{in}\quad O\times\{0\}.
\end{align}
Let $\mathbb{H}$ and $\mathbb{V}$ denote the divergent free Hilbert spaces given by:
\begin{align*} 
 &\mathbb{V}:=\left\{\textbf{v}\in\mathbb{H}_{0}^{1}(O,\mathbb{R}^2):\,\nabla\cdot\textbf{v}=0\,\, \mbox{in }\,O\right\}
\qquad\mbox{with}\qquad
\|\textbf{y}\|_{\mathbb{V}}=\left(\nabla\textbf{y},\nabla\textbf{y}\right)_{\mathbb{L}^2}\\
&\mbox{and} \qquad
 \mathbb{H}:= \mbox{closure of }{\mathbb{V}}\quad \mbox{in}\quad \mathbb{L}^2\quad \mbox{norm},\qquad
 \|\textbf{y}\|_{\mathbb{H}}=(\textbf{y},\textbf{y})_{\mathbb{L}^2}=|\textbf{y}|.
\end{align*}
Let $P$ be the Helmholtz-Hodge projection
\begin{align*} \label{25.04.15.E5}
P:\mathbb{L}^2(O,\mathbb{R}^{2})\rightarrow \mathbb{H},
\end{align*}
and $A:\mathbb{H}^2(O,\mathbb{R}^{2})\cap \mathbb{V} \to\mathbb{H},$ (Stokes operator) be defined by 
\begin{align}
A\textbf{y}=-\mu P\Delta\textbf{y}.
\end{align}
For additional information regarding these spaces and the operator $A,$ we refer Temam \cite{Temam}.\\
 In the Hilbert space $\mathbb{H}$, the  system \eqref{21.02.16.E1}-\eqref{28.03.16.E12} is given as
\begin{align*} 
\frac{d\textbf{y}}{dt}+ A \textbf{y}(t)+\int_{0}^{t}\beta(t-s)A\textbf{y}(s)ds&=P(u \,\chi_{O_0}) \quad\mbox{for all}\quad t > 0,\\
\textbf{y}(0)=\textbf{y}_{0},\quad \beta(0)=1.
\end{align*}
Using Riesz-Fredholm theory, we can conclude that that $A$ has a countable set of real eigenvalues 
$\lambda_{j}$ each of which is of finite algebraic multiplicity and  corresponding
set of eigenvectors $\phi_{j}$, that is,
\begin{align*} 
A\phi_{j}&= \lambda_{j} \phi_{j}\qquad j=1,2,...\\
\mbox{with}\quad \lambda_{j}&\to\infty,\quad j\to \infty.
\end{align*}
and for each $\lambda_{j},$ there is a finite number $m_{j}$ of linearly independent eigenvectors 
$\{\phi_{j}^{i}\}_{i=1}^{m_{j}}$ where $m_{j}$
is called the multiplicity of  $\lambda_{j}.$ As $A$ is self-adjoint, we note that $\{\phi_{n}\}_{n\in\mathbb{N}}$
forms an orthonormal basis in $\mathbb{H}$.
Now using the theory established in the previous sections, one can prove existence of finite dimensional real controller for  \eqref{21.02.16.E1}-\eqref{28.03.16.E12} analogus to Theorem \ref{28.02.16.T1} and Theorem \ref{15.02.16.T10}. We state below our main results of this section.
\begin{thm} \label{21.02.16.T1}
Let $\alpha \in [\frac{1}{2},\frac{3}{4}].$ Then for each $\gamma,\,0<\gamma<\delta$ and $M,$ as before, there is a linear
self-adjoint operator $R:D(R)\subset \mathbb{H}\rightarrow \mathbb{H}$ such that for some constants $0 <a_{1} <a_{2} <\infty,\,\, C_{1}>0$ 
we have:
\begin{itemize}
\item[(i)] The following equivalence inequality holds true:
\begin{align*} 
a_{1}|A^{\alpha-\frac{1}{2}}\textbf{y}|^{2}\leqslant (R\textbf{y},\textbf{y})
\leqslant a_{2}|A^{\alpha-\frac{1}{2}}\textbf{y}|^{2} 
\quad\mbox{for all}\quad \textbf{y} \in D(A^{\alpha-\frac{1}{2}}).
\end{align*}
\item[(ii)] The operator $R:D(A^{2\alpha-1}) \rightarrow \mathbb{H}$ is bounded, i.e.,
\begin{align*}
|R\textbf{y}| \leqslant C_{1}|A^{2\alpha-1}\textbf{y}| \quad\mbox{for all}\quad \textbf{y} \in D(A^{2\alpha-1}).
\end{align*}
\item[(iii)] $R$ satisfies the following algebraic Riccati equation:
\begin{align*}
\left( A \textbf{y}+\int_{0}^{t}\tilde{\beta}(t-s)A\textbf{y}(s)
 \,ds-\gamma\textbf{y},R\textbf{y}\right)+\frac{1}{2}\sum_{i=1}^{M}\left(\psi_{i},R\textbf{y}\right)_{0}^{2}
 =\frac{1}{2}|A^{\alpha}\textbf{y}|^2 \\
 \mbox{for all}\quad \textbf{y} \in D(A) 
\end{align*}
\end{itemize}
where $\tilde{\beta}(t)=e^{-(\delta-\gamma)t}.$
Moreover the Feedback controller
\begin{align*}
u^{*}(t)=-P\left(m\sum_{j=1}^{M}(\psi_{j},R\textbf{y}(t))_{0}\psi_{j}\right)
\end{align*}
exponentially stabilizes the linear system
\begin{align*} 
 \frac{d\textbf{y}}{dt}+ A \textbf{y}(t)+\int_{0}^{t}\beta(t-s)A\textbf{y}(s)
 \,ds&=u^{*}(t),t\geqslant 0\\
 \textbf{y}(0)&=\textbf{y}_0
\end{align*}
that is, the solution $\textbf{y}^*(t)$ to the corresponding closed loop system satisfies
\begin{align*}
|A^{\alpha-\frac{1}{2}}\textbf{y}^*(t)| &\leqslant \e^{-\gamma t}|A^{\alpha-\frac{1}{2}}\textbf{y}_0|
\end{align*}
and
\begin{align*}
\int_{0}^{\infty}\e^{2\gamma t}|A^{\alpha}\textbf{y}^*(t)|^2dt&\leqslant C|A^{\alpha-\frac{1}{2}}\textbf{y}_0|^2
\end{align*}
\end{thm}

where $(.,.)_{o}$ denotes the inner product in $(L^{2}(O_{0}))^{d}.$\\

 \begin{thm} \label{03.04.16.T2}
 Let $\alpha \in [0,\frac{1}{2}).$ Then for each $\gamma,\,0<\gamma<\delta $ and $M$ (as in $(\ref{03.03.16.E1})$),
 there is a bounded linear self-adjoint positive semidefinite operator $\tilde{R}:\mathbb{H}\rightarrow \mathbb{H}$ such that $\tilde{R}\in\mathcal{L}(\mathbb{H},D(A))$  
and $\tilde{R}$ satisfies the following algebraic Riccati equation:
\begin{align*}
\left( A \textbf{y}+\int_{0}^{t}\tilde{\beta}(t-s)A\textbf{y}(s)
 \,ds-\gamma\textbf{y},\tilde{R}\textbf{y}\right)+\frac{1}{2}\sum_{i=1}^{M}\left(\psi_{i},R\textbf{y}\right)_{0}^{2}
 =\frac{1}{2}|A^{\alpha}\textbf{y}|^2 \\
 \mbox{for all}\quad \textbf{y} \in D(A), 
\end{align*}
where $\tilde{\beta}(t)=e^{-(\delta-\gamma)t}.$
Moreover the feedback controller
\begin{align*} 
u^{*}(t)=-P\left(m\sum_{j=1}^{M}(\psi_{j},R\textbf{y}(t))_{0}\psi_{j}\right)
\end{align*}
exponentially stabilizes the linear system
\begin{align*} 
 \frac{d\textbf{y}}{dt}+ A \textbf{y}(t)+\int_{0}^{t}\beta(t-s)A\textbf{y}(s)
 \,ds&=u^*(t),t\geqslant 0.\\
 \textbf{y}(0)&=\textbf{y}_0
\end{align*}
that is, the solution $\textbf{y}^*$ to the corresponding closed loop system satisfies
\begin{align*} 
|\textbf{y}^*(t)| &\leqslant C\e^{-\gamma t}|\textbf{y}_0|,
\end{align*}
where $(.,.)_{o}$ denotes the inner product in $(L^{2}(O_{0}))^{d}.$
\end{thm}

The proof of above theorems is a direct application of Theorem \ref{26.12.15.T1} and Theorem \ref{03.04.16.T1}.

\subsection{Application to Jeffreys fluid}
Let $\Omega\subset\mathbb{R}^2$ be a bounded domain with $\partial \Omega\in C^2.$ Now we consider the following 
system for $T\in(0,\infty]$ for the velocity vector $\textbf{y},$ pressure $p$ and the fluid stress tensor $\tau$ of a viscoelastic
Jeffreys fluid model:
\begin{align}\label{26.03.16.E1}
 \partial_t\textbf{y}-\mu\Delta\textbf{y}+\nabla p&=\nabla\cdot \tau+ u\chi_{O_0}\quad\mbox{in}\quad\Omega\times(0,T)\\
 \nabla\cdot\textbf{y}&=0\quad\mbox{in}\quad\Omega\times(0,T)\\\label{28.03.16.E6}
 \partial_t\tau+\lambda\tau&=2\kappa D\textbf{y}\quad\mbox{in}\quad\Omega\times(0,T)\\
 \textbf{y}&=0\quad\mbox{on}\quad\partial\Omega\times(0,T)\\\label{28.03.16.E7}
 \textbf{y}(\cdot,0)=\textbf{y}_0&\qquad \tau(\cdot,0)=\tau_0 \quad\mbox{in}\quad\Omega.
\end{align}
where $\mu,\lambda$ and $\kappa$ are positive constants and $D\textbf{y}$ is the symmetrized gradient
tensor defined by $$D\textbf{y}:=\frac{1}{2}\left(\nabla\textbf{y}+\nabla^t\textbf{y}\right).$$ For additional information on the physical meaning of these parameters,
see for instance Renardy et al. \cite{RHN}, Joseph \cite{Jos}.  For Jeffreys fluid, approximate controllability results have been proved in Chowdhury et al. \cite{Chow}.

Note that from the above equation $\tau$ can be written as
\begin{align}\label{26.03.16.E2}
 \tau(t)=e^{-\lambda t}\tau_0+2\kappa\int_0^t e^{-\lambda(t-s)}D\textbf{y}(s)ds\quad\forall\,t>0,
\end{align}
Using \eqref{26.03.16.E2}, and Helmholtz-Hodge projection, the system \eqref{26.03.16.E1})-\eqref{28.03.16.E7} becomes
\begin{align*}
 \partial_t\textbf{y}+A\textbf{y}+\frac{\kappa}{\mu}\int_0^t e^{-\lambda(t-s)}A\textbf{y}(s)ds
 &=e^{-\lambda t}\nabla\cdot \tau_0+ P(u\chi_{O_0})\quad\mbox{for all}\ t>0,\\
 \textbf{y}(0)&=\textbf{y}_0.
\end{align*}

The viscoelastic fluids of the Jeffreys kind, can be used as first approximations (taking into account that $\textbf{y}, \tau$ are small) of the  nonlinear system, \eqref{25.04.15.E1}-\eqref{28.03.15.E9}, see Doubova et al. \cite{Dou}, Joseph \cite{Jos}. Thus, results analogous to Theorem \ref{21.02.16.T1}, Theorem \ref{03.04.16.T2} hold true for Jeffreys fluid as well.

\section{Further Remarks}
Many interesting questions arise which are still open for Oldroyd fluid in particular and for abstract PIDE in general. The immediate extension of above work which authors are interested to take up is
exponential stabilization of the Oldroyd fluid model \eqref{25.04.15.E1}-\eqref{28.03.15.E9} around unstable non-zero steady state solution by means of feedback controller. 
In the current work we have studied stabilization via interior control. Similar questions are interesting when control is applied on the boundary. Thus authors wish to investigate exponential stabilization of the abstract PIDE around unstable stationary solution via boundary control in feedback form and 
exponential stabilization of  Oldroyd fluid model and Jeffreys fluid model
around  unstable  stationary solution, by means of a feedback boundary control.

\medskip\noindent
{\bf Acknowledgements:} Utpal Manna's work has been supported by National
Board of Higher Mathematics (NBHM), Govt. of India. All
the authors would like to thank Indian Institute of Science Education
and Research Thiruvananthapuram for providing stimulating scientific environment and resources. Authors would like to thank Amiya K Pani from IIT Bombay for useful discussions.




\end{document}